\documentclass{article}

\usepackage{arxiv}

\usepackage[utf8]{inputenc} % allow utf-8 input
\usepackage[T1]{fontenc}    % use 8-bit T1 fonts
\usepackage{hyperref}       % hyperlinks
\usepackage{url}            % simple URL typesetting
\usepackage{booktabs}       % professional-quality tables
\usepackage{amsfonts}       % blackboard math symbols
\usepackage{nicefrac}       % compact symbols for 1/2, etc.
\usepackage{microtype}      % microtypography
\usepackage{graphicx}
\usepackage{enumerate}
\usepackage[shortlabels]{enumitem}
\usepackage{longtable}
\usepackage{amsmath,amssymb,amsfonts,amsthm}
\usepackage{algorithm}
\usepackage{comment}
\usepackage{xcolor}
\usepackage[lofdepth,lotdepth]{subfig}
\usepackage{pgf,tikz}
\usepackage{pgfplots}

\newtheorem{theorem}{Theorem}
\newtheorem{cor}{Corollary}
\newtheorem{lemma}{Lemma}
\newtheorem{defn}{Definition}
\newtheorem{remark}{Remark}

\newcommand{\Int}{\mathrm{int}}
\title{Cross-points in the Dirichlet-Neumann method II: a geometrically convergent variant}

\author{
 Bastien Chaudet-Dumas \\
  Section of Mathematics\\
  University of Geneva\\
  \texttt{bastien.chaudet@unige.ch} \\
  \And
 Martin J. Gander \\
  Section of Mathematics\\
  University of Geneva\\
  \texttt{martin.gander@unige.ch} \\  
}

\begin{document}

\maketitle

\begin{abstract}
    When considered as a standalone iterative solver for elliptic boundary value problems, the Dirichlet-Neumann (DN) method is known to converge geometrically for domain decompositions into strips, even for a large number of subdomains. However, whenever the domain decomposition includes cross-points, i.e.$\!$ points where more than two subdomains meet, the convergence proof does not hold anymore as the method generates subproblems that might not be well-posed. 
    Focusing on a simple two-dimensional example involving one cross-point, we proposed in \cite{chaudet2022cross1} a decomposition of the solution into two parts: an even symmetric part and an odd symmetric part. Based on this decomposition, we proved that the DN method was geometrically convergent for the even symmetric part and that it was not well-posed for the odd symmetric part.
    Here, we introduce a new variant of the DN method which generates subproblems that remain well-posed for the odd symmetric part as well. Taking advantage of the symmetry properties of the domain decomposition considered, we manage to prove that our new method converges geometrically in the presence of cross-points. We also extend our results to the three-dimensional case, and present numerical experiments that illustrate our theoretical findings.
\end{abstract}

% keywords can be removed

% --------------------
\section{Introduction} 
\label{sec:Intro}
% --------------------

% Small intro DN + Cross-points.
The Dirichlet-Neumann (DN) method was first introduced in \cite{bjorstad1984solving,Bjorstad:1986:IMS}, making it one of the first non-overlapping domain decomposition methods. Even though it has been noted as early as in \cite{de1990analysis} that the method might not be well-posed in the presence of cross-points, this issue does not seem to have received a lot of attention since then. Indeed, the main research effort was first put on the analysis of the DN method as a preconditioner to accelerate Krylov solvers, see \cite{bjorstad1984solving,Bjorstad:1986:IMS,bramble1986iterative,widlund1986method,dryja1988method,marini1989relaxation} and also \cite{chan1994domain,Toselli:2004:DDM} for more general reviews on domain decomposition methods. In this context, the problematic behaviour of the DN method is masked by the efficiency of the Krylov solver. Besides, many authors, who use the DN method to solve more complex model problems, consider configurations without cross-points such as partitions with two subdomains \cite{maier2014dirichlet,song2021analysis,chaouqui2023nonlinear} or with many subdomains into strips \cite{chaouqui2018scalability,gander2021dirichlet}.

% Description of part I
In the first part of this work \cite{chaudet2022cross1}, we provided a complete analysis of the DN method in a simple 2D configuration containing one cross-point: a square divided into four squared subdomains. By decomposing the solution into its even and odd symmetric parts and conducting separate analyses for these two cases, we managed to prove two important results. First, the DN method is geometrically convergent when dealing with the even symmetric part. Second, the DN method is not well-posed when dealing with the odd symmetric part. More precisely, the iterative process generates (non unique) solutions that are singular in the neighbourhood of the cross-point.

% Aim of part II and novelty compared to what has been done so far
In this second part, our goal is to use this new knowledge on how the original DN method behaves with even/odd symmetric functions, in order to develop new transmission conditions specifically designed to treat the odd symmetric part. Our study reveals that introducing a new distribution of Dirichlet and Neumann transmission conditions enables us to recover the geometric convergence for the odd symmetric part. To the best of our knowledge, this is the first proof of geometric convergence at the continuous level for a non-overlapping domain decomposition method based on local transmission conditions in a configuration with cross-points, apart from our previous work \cite{chaudet2023cross} on the Neumann-Neumann method. The first optimized Schwarz methods, which were also based on local transmission conditions, have been proved to be convergent, but not geometrically. We refer to \cite{Lions:1990:SAM} for the positive definite elliptic case and \cite{Despres:1991:MDD} for the Helmholtz problem. In the case of the Helmholtz problem with cross-points, convergence (not necessarily geometric) has also been demonstrated when using higher order local transmission conditions \cite{modave2020non,despres2022optimized}. Finally, in other recent works dedicated to the Hemholtz problem, geometric convergence has been obtained for optimized Schwarz methods, but using non local transmission conditions \cite{claeys2022robust,collino2020exponentially}. 

% Outline of the manuscript
This manuscript is organized as follows. In section \ref{sec:Preliminaries}, we present the model problem and gather some necessary results on the regularity of the solution to the Laplace problem in polygons. In section \ref{sec:NewDN}, we introduce our new variant of the DN method, and show that it converges geometrically for the even and odd symmetric parts. Then, in section \ref{sec:3D}, we extend this convergence result to the 3D case of a cube divided into four subdomains. Finally, we show some 2D and 3D numerical results in section \ref{sec:Numerics}.

% ---------------------------------------------
\section{Model problem and preliminary notions} 
\label{sec:Preliminaries}
% ---------------------------------------------

As in \cite{chaudet2022cross1},
\begin{figure}
  \centering
  \subfloat[Standard DNM.]{  
  \resizebox{0.27\textwidth}{!}{  
  \begin{tikzpicture}

    \draw[black, thin] (-2,-2) -- (2,-2) -- (2,2) -- (-2, 2) -- (-2,-2);
    \draw[black, densely dashed] (0,-2) -- (0,2); 
    \draw[black, densely dashed] (-2,0) -- (2,0); 
    
    \fill[gray, opacity=0.4] (0,-2) -- (2,-2) -- (2,0) -- (0,0) -- cycle; 
    \fill[gray, opacity=0.4] (-2,0) -- (0,0) -- (0,2) -- (-2,2) -- cycle; 
    
    \node[] at (0,-2) {\tiny{$\bullet$}};
    \node[] at (0,2) {\tiny{$\bullet$}};
    \node[] at (-2,0) {\tiny{$\bullet$}};
    \node[] at (2,0) {\tiny{$\bullet$}};
    \node[red] at (0,0) {\tiny{$\bullet$}};
    
    \node[] at (-1.25,-1.25) {\small{$\Omega_1$}};
    \node[] at (1.25,-1.25) {\small{$\Omega_2$}};
    \node[] at (1.25,1.25) {\small{$\Omega_3$}};
    \node[] at (-1.25,1.25) {\small{$\Omega_4$}};

    \node[] at (-0.3,-1) {\small{\textbf{D}}};
    \node[] at (-1,-0.3) {\small{\textbf{D}}};
    \node[] at (0.3,1) {\small{\textbf{D}}};
    \node[] at (1,0.3) {\small{\textbf{D}}};
    \node[] at (0.3,-1) {\small{\textbf{N}}};
    \node[] at (-1,0.3) {\small{\textbf{N}}};
    \node[] at (-0.3,1) {\small{\textbf{N}}};
    \node[] at (1,-0.3) {\small{\textbf{N}}};
    
  \end{tikzpicture}
  }
  \label{sub:SchDDstd}
  }
  \hspace{3em}
  \subfloat[New variant of the DNM.]{
  \resizebox{0.27\textwidth}{!}{  
  \begin{tikzpicture}

    \draw[black, thin] (-2,-2) -- (2,-2) -- (2,2) -- (-2, 2) -- (-2,-2);
    \draw[black, densely dashed] (0,-2) -- (0,2); 
    \draw[black, densely dashed] (-2,0) -- (2,0); 

    \fill[gray, opacity=0.4] (0,-2) -- (2,-2) -- (2,0) -- (0,0) -- cycle; 
    \fill[gray, opacity=0.4] (-2,0) -- (0,0) -- (0,2) -- (-2,2) -- cycle; 
        
    \node[] at (0,-2) {\tiny{$\bullet$}};
    \node[] at (0,2) {\tiny{$\bullet$}};
    \node[] at (-2,0) {\tiny{$\bullet$}};
    \node[] at (2,0) {\tiny{$\bullet$}};
    \node[red] at (0,0) {\tiny{$\bullet$}};
    
    \node[] at (-1.25,-1.25) {\small{$\Omega_1$}};
    \node[] at (1.25,-1.25) {\small{$\Omega_2$}};
    \node[] at (1.25,1.25) {\small{$\Omega_3$}};
    \node[] at (-1.25,1.25) {\small{$\Omega_4$}};

    \node[] at (-0.3,-1) {\small{\textbf{N}}};
    \node[] at (-1,-0.3) {\small{\textbf{D}}};
    \node[] at (0.3,1) {\small{\textbf{N}}};
    \node[] at (1,0.3) {\small{\textbf{D}}};
    \node[] at (0.3,-1) {\small{\textbf{D}}};
    \node[] at (-1,0.3) {\small{\textbf{N}}};
    \node[] at (-0.3,1) {\small{\textbf{D}}};
    \node[] at (1,-0.3) {\small{\textbf{N}}};
    
  \end{tikzpicture}
  }
  \label{sub:SchDDnew}  
  }
  \caption{Domain $\Omega$ divided into four square subdomains with transmission conditions of Dirichlet (\textbf{D}) or Neumann (\textbf{N}) type.}
  \label{fig:SchDD}
\end{figure}
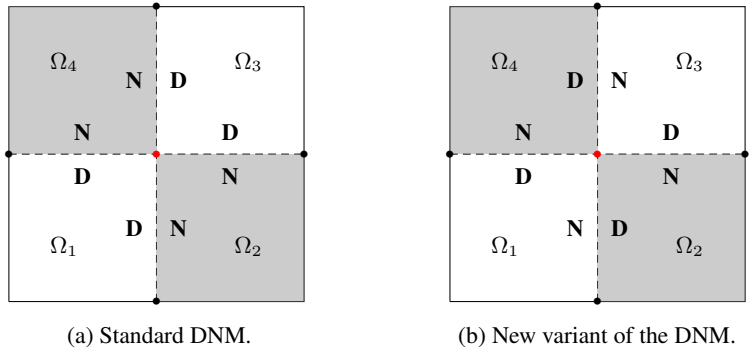
we consider for the domain $\Omega\subset \mathbb{R}^2$ the square $(-1,1)\times(-1,1)$. We divide it into four
non-overlapping square subdomains $\Omega_i$, $i\in
\mathcal{I}:=\{1,2,3,4\}$, of equal area, see Figure \ref{fig:SchDD}. With such a partition, there
is one interior cross-point (red dot), and there are also four
boundary cross-points (black dots). We denote
the interfaces between adjacent subdomains by $\Gamma_{ij}:=
\Int(\partial\Omega_i\cap\partial\Omega_j)$, and the skeleton of the
partition by $\Gamma:=\bigcup_{i,j} \overline{\Gamma}_{ij}$. Here, we need to artificially distinguish
between $\Gamma_{ij}$ and $\Gamma_{ji}$ for all $i$, $j$. More specifically, for each $(i,j)$, the same part of the boundary is referred to as $\Gamma_{ij}$ from the viewpoint of $\Omega_i$, and $\Gamma_{ji}$ from the viewpoint of $\Omega_j$. We further denote the
interior of the intersection between $\partial\Omega_i$ and the
boundary $\partial\Omega$ by $\partial\Omega_i^0:=\Int(\partial\Omega_i\cap\partial\Omega)$, and
the interior of the left, right, bottom and top sides of $\Omega$ by
$\partial\Omega_l$, $\partial\Omega_r$, $\partial\Omega_b$, $\partial\Omega_t$. Thus an arbitrary side
of $\Omega$ is denoted by $\partial\Omega_\sigma$, where $\sigma$ is in
the set of indices $\mathcal{S}:=\{b,r,t,l\}$. We also use the same
notation for an arbitrary side of a subdomain $\Omega_i$, namely
$\partial\Omega_{i,\sigma}$ with $\sigma\in \mathcal{S}$.

\subsection{Laplace model problem}
\label{subsec:ModelPb}

We consider the Laplace problem in $\Omega$ with Dirichlet, Robin, or mixed boundary conditions, that is: find $u$ solution to
\begin{equation}
    \left\{
    \begin{aligned}
        -\Delta u &= f \: \mbox{ in } \Omega , \\
        \mathcal{B}u &= g\: \mbox{ on } \partial\Omega ,
    \end{aligned}
    \right.
    \label{eqn:ModelPb}
\end{equation}
where the boundary condition is piecewise defined on $\partial\Omega$ as follows:
\begin{equation*}
	\begin{aligned}
		 u &= g^{\mathcal{D}} \: \mbox{ on } \partial\Omega_{\sigma} \mbox{ if } \sigma\in\mathcal{D}, \\
		 (\partial_n+p)u &= g^{\mathcal{R}}\: \mbox{ on } \partial\Omega_{\sigma} \mbox{ if } \sigma\in\mathcal{R},
	\end{aligned}
\end{equation*}
with $\mathcal{D}$ and $\mathcal{R}$ two (possibly empty) subsets of $\mathcal{S}$ such that $\mathcal{D}\cup\mathcal{R}=\mathcal{S}$. We make the following usual assumption on the regularity of the data.

\vspace{0.5em}
\textbf{Assumption 1.} We assume that $f\in H^{-1}(\Omega)$, $g^{\mathcal{D}}\in H^{\frac{1}{2}}(\partial\Omega)$,  $p\in L^\infty(\partial\Omega)$ ($p\geq 0$ a.e. on $\partial\Omega$), and $g^{\mathcal{R}}\in
H^{-\frac12}(\partial\Omega)$. Moreover, if $\mathcal{D}=\emptyset$, then $p$ is strictly positive on a subset of
$\partial\Omega$ of non-zero measure.
\vspace{0.5em}

Since $\Omega$ is Lipschitz, it is known that \eqref{eqn:ModelPb} admits a unique solution $u\in H^1(\Omega)$ when Assumption 1 holds.

\subsection{Regularity results}

In this study, we are interested in pointwise properties of the functions we deal with. In particular, we would like to work with solutions that are continuous on $\overline{\Omega}$, and whose normal derivatives are defined at least almost everywhere on the boundary. Such properties are obviously not satisfied if $u\in H^1(\Omega)$, but they are if $u\in H^2(\Omega)$ since $\Omega\subset\mathbb{R}^2$.

In the context of polygons, regardless of the type of boundary conditions, regularity results for elliptic PDEs are more technical when considering $H^m(\Omega)$ regularity ($m\geq 2$) compared to $W^{m,p}(\Omega)$ regularity for $p>2$. This is due to the fact that functions in $H^{\frac12}(\partial\Omega_\sigma)$ are not necessarily continuous, while the Sobolev embedding theorem guarantees that functions in $W^{1-\frac1p,p}(\partial\Omega_\sigma)$ are always continuous on $\overline{\partial\Omega}_\sigma$ for $p>2$. Therefore the compatibility relations, which are often expressed in a rather technical way in the $H^m$ case (see e.g. \cite[Theorem 4]{chaudet2022cross1}), are simply expressed as pointwise conditions in the $W^{m,p}$ case, see again \cite[Chapter 1]{grisvard2011elliptic}. This is why we are interested in $W^{2,p}$ regularity for some fixed $p>2$ rather than the more conventional $H^2$ regularity. 

In what follows, we will express sufficient conditions on the data for the solution to \eqref{eqn:ModelPb} to belong to $W^{2,p}(\Omega)$. In the model problem considered here, the boundary conditions are slightly more general than the ones considered in the first part \cite{chaudet2022cross1}. Namely we allow mixed boundary conditions of Dirichlet-Robin type, which are less standard than mixed conditions of Dirichlet-Neumann type. This implies that we need to adapt the regularity results from \cite{grisvard2011elliptic}. We consider three different configurations: the full Dirichlet case ($\mathcal{R}=\emptyset$), the full Robin case ($\mathcal{D}=\emptyset$) and the mixed Dirichlet-Robin case ($\mathcal{D}\neq\emptyset$, $\mathcal{R}\neq\emptyset$). Let us recall that we are dealing with a polygonal domain $\Omega$, thus the regularity on the whole boundary $\partial\Omega$ can be expressed by means of the regularity on each side of the polygon associated with compatibilty relations at the vertices. In the rectilinear case, vertices are actually corners, which we gather in the set $\mathcal{C}:=\left\{ (l,b), (b,r), (r,t), (t,l) \right\} $ consisting of four pairs of indices in $\mathcal{S}$.
As in \cite{chaudet2022cross1}, we split $\mathcal{C}$ into three disjoint subsets $\mathcal{C}=\mathcal{C}_{\mathcal{D}}\cup\mathcal{C}_{\mathcal{R}}\cup\mathcal{C}_{\mathcal{M}}$ corresponding to Dirichlet corners, Robin corners and mixed corners. More specifically, we have for each $(\sigma,\sigma')\in\mathcal{C}$,
\begin{equation*}
	(\sigma,\sigma') \in 
	\left\{
    \begin{aligned}
         \mathcal{C}_{\mathcal{D}} \: & \mbox{ if } \sigma,\sigma' \in \mathcal{D}, \\
         \mathcal{C}_{\mathcal{R}} \: & \mbox{ if } \sigma,\sigma' \in \mathcal{R}, \\
         \mathcal{C}_{\mathcal{M}} \: & \mbox{ if } \sigma \in \mathcal{D}, \sigma' \in \mathcal{R}.
    \end{aligned}
    \right.
\end{equation*}

\textbf{Notation.} For any function $h$ defined on $\partial\Omega$, we denote its restriction to a part of the boundary $\partial\Omega_\sigma$ by $h_\sigma:=h\vert_{\partial\Omega_\sigma}$, for each $\sigma\in\mathcal{S}$. Similarly, its normal derivative on $\partial\Omega_\sigma$ will be denoted by $\partial_{n_\sigma}h := \nabla h\vert_{\partial\Omega_\sigma}\cdot \,n_\sigma$. Furthermore, for any corner $(\sigma,\sigma')\in\mathcal{C}$, we denote the affine geometric transformations from the unit segment $(0,1)$ to $\partial\Omega_\sigma$ and $\partial\Omega_{\sigma'}$ by
\begin{equation*}
	\begin{aligned}
		\mathrm{x}^\sigma_{\sigma,\sigma'} : (0,1) &\longrightarrow \partial\Omega_\sigma \\
		 s &\longmapsto \mathrm{x}^\sigma_{\sigma,\sigma'}(s)
	\end{aligned}
	\qquad \mbox{ and } \qquad
	\begin{aligned}
		\mathrm{x}^{\sigma'}_{\sigma,\sigma'} : (0,1) &\longrightarrow \partial\Omega_{\sigma'} \\
		 s &\longmapsto \mathrm{x}^{\sigma'}_{\sigma,\sigma'}(s)
	\end{aligned}
\end{equation*}
such that the coordinates of the vertex at $(\sigma,\sigma')$ are given by $\mathrm{x}^{\sigma}_{\sigma,\sigma'}(0)=\mathrm{x}^{\sigma'}_{\sigma,\sigma'}(0)$. In what follows, when it is clear from the context that we work in the neighbourhood of some vertex $(\sigma,\sigma')$, we often refer to the function $h_\sigma\circ \mathrm{x}^\sigma_{\sigma,\sigma'} : (0,1)\rightarrow \mathbb{R}$ simply as $h_\sigma$. This will enable us to write equalities such as $h_\sigma=z_{\sigma'}$, meaning that the functions $h_\sigma\circ \mathrm{x}^\sigma_{\sigma,\sigma'}$ and $z_{\sigma'}\circ \mathrm{x}^{\sigma'}_{\sigma,\sigma'}$ are equal a.e. on $(0,1)$. As a result, when $h_\sigma$ is sufficiently smooth and can be extended to $\overline{\partial\Omega}_\sigma$, the scalar $h_\sigma(0)$ denotes the value of $h_\sigma$ at the vertex $(\sigma,\sigma')$.

\subsubsection{Full Dirichlet and full Robin cases}

These cases were already considered in \cite[Chapter 1]{grisvard2011elliptic}, so we just recall the corresponding regularity results.
\begin{theorem}[Dirichlet]
	If in addition to Assumption 1, we have $f\in L^p(\Omega)$, $g^{\mathcal{D}}_\sigma\in W^{2-\frac{1}{p},p}(\partial\Omega_\sigma)$ for all $\sigma\in \mathcal{S}$, and $g^{\mathcal{D}}_\sigma(0)=g^{\mathcal{D}}_{\sigma'}(0)$ for all $(\sigma,\sigma')\in\mathcal{C}$, then the solution $u$ to \eqref{eqn:ModelPb} is in $W^{2,p}(\Omega)$.
    \label{thm:RegDir}
\end{theorem}
\begin{theorem}[Robin]
	If in addition to Assumption 1, we have $f\in L^p(\Omega)$, $p_\sigma\in C^\infty(\overline{\partial\Omega}_\sigma)$ and $g^{\mathcal{R}}_\sigma\in W^{1-\frac{1}{p},p}(\partial\Omega_\sigma)$ for all $\sigma\in \mathcal{S}$, then the solution $u$ to \eqref{eqn:ModelPb} is in $W^{2,p}(\Omega)$.
    \label{thm:RegRob}
\end{theorem}

\subsubsection{Mixed Dirichlet-Robin case}

The specific case of mixed Dirichlet-Robin boundary conditions has not received a lot of attention. We mention \cite{mghazli1992regularity}, where the author proves that the structure of the solution in the Dirichlet-Robin case is the same as the one in the mixed Dirichlet-Neumann case obtained in \cite[Chapter 4]{grisvard2011elliptic}. However, no expressions for the compatibility relations ensuring existence of a more regular solution are provided. In this subsection, our goal is to obtain such relations in the case of $W^{2,p}$ regularity.

Before doing so, let us recall the regularity result from \cite[Chapter 4]{grisvard2011elliptic} in the case of mixed Dirichlet-Neumann homogeneous boundary conditions.
\begin{theorem}[Homogeneous Dirichlet-Neumann]
	Let us consider the problem
	\begin{equation}
    		\left\{
    		\begin{aligned}
        		-\Delta v &= k \: \mbox{ in } \Omega , \\
        		v &= 0\: \mbox{ on } \partial\Omega_\sigma \mbox{ if } \sigma\in\mathcal{D}, \\
        		\partial_n v &= 0\: \mbox{ on } \partial\Omega_\sigma \mbox{ if } \sigma\in\mathcal{R}.
    		\end{aligned}
    		\right.
    		\label{eqn:HomogDNPb}
	\end{equation}
	If we assume that $k\in L^p(\Omega)$, then the solution $v$ to \eqref{eqn:HomogDNPb} is in $W^{2,p}(\Omega)$.
\label{thm:RegHomogDNPb}
\end{theorem} 
The case of mixed Dirichlet-Robin homogeneous boundary conditions, which will be useful in the proof of the main result of this subsection, can be deduced from the previous result.
\begin{cor}[Homogeneous Dirichlet-Robin]
	Let us consider the problem
	\begin{equation}
    		\left\{
    		\begin{aligned}
        		-\Delta v &= k \: \mbox{ in } \Omega , \\
        		v &= 0\: \mbox{ on } \partial\Omega_\sigma \mbox{ if } \sigma\in\mathcal{D}, \\
        		(\partial_n + p) v &= 0\: \mbox{ on } \partial\Omega_\sigma \mbox{ if } \sigma\in\mathcal{R}.
    		\end{aligned}
    		\right.
    		\label{eqn:HomogPb}
	\end{equation}
	If, in addition to the assumptions of subsection \ref{subsec:ModelPb}, we assume that $k\in L^p(\Omega)$ and $p_\sigma\in C^\infty(\overline{\partial\Omega}_\sigma)$ for all $\sigma \in \mathcal{R}$, then the solution $v$ to \eqref{eqn:HomogPb} is in $W^{2,p}(\Omega)$.
\label{cor:RegHomogPb}
\end{cor} 

\begin{proof}
	We already know from standard results that there exists a unique $v\in W^{1,p}(\Omega)$ solution to \eqref{eqn:HomogPb}. Now, if $v$ solves \eqref{eqn:HomogPb}, it also solves a mixed Dirichlet-Neumann problem with non-homogeneous boundary conditions:
		\begin{equation}
    		\left\{
    		\begin{aligned}
        		-\Delta v &= k \: \mbox{ in } \Omega , \\
        		v &= 0\: \mbox{ on } \partial\Omega_\sigma \mbox{ if } \sigma\in\mathcal{D}, \\
        		\partial_n v &= -pv \: \mbox{ on } \partial\Omega_\sigma \mbox{ if } \sigma\in\mathcal{R}.
    		\end{aligned}
    		\right.
    		\label{eqn:NonHomogDNPb}
	\end{equation}
	The following step is to find a extension for these boundary conditions in $W^{2,p}(\Omega)$, which will enable us to recover the homogeneous case and apply Theorem \ref{thm:RegHomogDNPb}. In order to do so, we use the compatibility relations for mixed Dirichlet-Neumann conditions provided by \cite[Theorem 1.5.2.8]{grisvard2011elliptic}. More specifically, we know that a function $w$ belongs to $W^{2,p}(\Omega)$ if and only if its traces and normal derivatives satisfy
	\begin{subequations} \label{eqn:CompRel:all}
	\begin{align}
		\left( w_\sigma, \partial_{n_\sigma}w \right)& \in W^{2-\frac1p,\frac1p}(\partial\Omega_\sigma)\times W^{1-\frac1p,\frac1p}(\partial\Omega_\sigma), \label{eqn:CompRel:1}\\
		w_\sigma(0)&=w_{\sigma'}(0) \mbox{ for all } (\sigma,\sigma')\in\mathcal{C}, \label{eqn:CompRel:2}\\
		(w_\sigma)'(0)&=\partial_{n_{\sigma'}}w(0) \mbox{ for all } (\sigma,\sigma')\in\mathcal{C}. \label{eqn:CompRel:3}
	\end{align}
	\end{subequations}	
%	\begin{itemize}
%		\item[(a)] $\left( w_\sigma, \partial_{n_\sigma}w \right) \in W^{2-\frac1p,\frac1p}(\partial\Omega_\sigma)\times W^{1-\frac1p,\frac1p}(\partial\Omega_\sigma)$,
%		\item[(b)] $w_\sigma(0)=w_{\sigma'}(0)$ for all $(\sigma,\sigma')\in\mathcal{C}$,
%		\item[(c)] $(w_\sigma)'(0)=\partial_{n_{\sigma'}}w(0)$ for all $(\sigma,\sigma')\in\mathcal{C}$.
%	\end{itemize}
	Here, we need to build pairs of functions $(h_\sigma,z_\sigma)$, representing the Dirichlet and Neumann traces on each $\partial\Omega_\sigma$, that satisfy these three conditions and verify the boundary conditions in \eqref{eqn:NonHomogDNPb}. Let $(\sigma,\sigma')\in\mathcal{C}$, we begin with building local Dirichlet and Neumann traces on each side of the vertex $(\sigma,\sigma')$, denoted by $(h_{\sigma,\sigma'}^\sigma$, $z_{\sigma,\sigma'}^\sigma)$ on $\partial\Omega_\sigma$ and $(h_{\sigma,\sigma'}^{\sigma'}$, $z_{\sigma,\sigma'}^{\sigma'})$ on $\partial\Omega_{\sigma'}$. These traces depend on the type of vertex considered:
\begin{itemize}
	\item if $(\sigma,\sigma')\in\mathcal{C}_{\mathcal{D}}$, we set
	\begin{equation*}
		h_{\sigma,\sigma'}^\sigma = 0, \quad h_{\sigma,\sigma'}^{\sigma'} = 0, \quad z_{\sigma,\sigma'}^\sigma = 0, \quad z_{\sigma,\sigma'}^{\sigma'} = 0;
	\end{equation*}
	\item if $(\sigma,\sigma')\in\mathcal{C}_{\mathcal{R}}$, we set for all $s\in(0,1)$
	\begin{equation*}
	\begin{aligned}
		h_{\sigma,\sigma'}^\sigma &= -(p_{\sigma'}v_{\sigma'})(0)s, \qquad & z_{\sigma,\sigma'}^\sigma &= -p_\sigma v_\sigma, \\
		h_{\sigma,\sigma'}^{\sigma'} &= -(p_\sigma v_\sigma)(0)s, \qquad & z_{\sigma,\sigma'}^{\sigma'} &= -p_{\sigma'} v_{\sigma'};
	\end{aligned}
	\end{equation*}
	\item if $(\sigma,\sigma')\in\mathcal{C}_{\mathcal{M}}$, say $\sigma\in\mathcal{D}$ and $\sigma'\in\mathcal{R}$, we set
	\begin{equation*}
		h_{\sigma,\sigma'}^\sigma = 0, \quad h_{\sigma,\sigma'}^{\sigma'} = 0, \quad z_{\sigma,\sigma'}^\sigma = 0, \quad z_{\sigma,\sigma'}^{\sigma'} = -p_{\sigma'}v_{\sigma'}.
	\end{equation*}	
\end{itemize}
Given the regularities of $p$ and $v$, the definitions above ensure that
\begin{equation}
	(h_{\sigma,\sigma'}^\sigma, z_{\sigma,\sigma'}^\sigma) \in W^{2-\frac1p,\frac1p}(\partial\Omega_\sigma)\times W^{1-\frac1p,\frac1p}(\partial\Omega_\sigma),
	\label{eqn:localregcheck}
\end{equation}
and the same regularity property holds for $(h_{\sigma,\sigma'}^{\sigma'}, z_{\sigma,\sigma'}^{\sigma'})$. Moreover, it can be checked that
\begin{equation}
	h_{\sigma,\sigma'}^\sigma(0) = h_{\sigma,\sigma'}^{\sigma'}(0), \quad
	(h_{\sigma,\sigma'}^\sigma)'(0) = z_{\sigma,\sigma'}^{\sigma'}(0), \quad
	(h_{\sigma,\sigma'}^{\sigma'})'(0) = z_{\sigma,\sigma'}^{\sigma}(0).
	\label{eqn:localCRcheck}
\end{equation}
Finally, we also have
\begin{equation}
	\left\{
	\begin{aligned}
		h_{\sigma,\sigma'}^\sigma &= 0 \mbox{ if } \sigma\in\mathcal{D}, \\
		z_{\sigma,\sigma'}^\sigma &= -p_\sigma v_\sigma \mbox{ if } \sigma\in\mathcal{R},
	\end{aligned}
	\right.
	\label{eqn:localHomogBCcheck}
\end{equation}
and the same holds for $(h_{\sigma,\sigma'}^{\sigma'}, z_{\sigma,\sigma'}^{\sigma'})$.
Now that we have defined traces on both sides of each vertex, we are left with two pairs of Dirichlet and Neumann traces on each side of the polygon. Indeed, for any $\sigma\in\mathcal{S}$, with neighbouring sides $\sigma'$ and $\sigma''$, we have one pair $(h_{\sigma,\sigma'}^\sigma, z_{\sigma,\sigma'}^\sigma)$ corresponding to the vertex $(\sigma,\sigma')$, and one pair $(h_{\sigma,\sigma''}^\sigma, z_{\sigma,\sigma''}^\sigma)$ corresponding to the other vertex $(\sigma,\sigma'')$. In order to recombine these traces, we introduce a $C^\infty$ partition of unity $(\psi_{\sigma,\sigma'}^\sigma, \psi_{\sigma,\sigma''}^\sigma)$ on $\partial\Omega_\sigma$ such that $\psi_{\sigma,\sigma'}^\sigma \equiv 1$ in the neighbourhood of $(\sigma,\sigma')$ and $\psi_{\sigma,\sigma''}^\sigma \equiv 1$ in the neighbourhood of $(\sigma,\sigma'')$. This enables us to define a single pair $(h_\sigma,z_\sigma)$ of Dirichlet and Neumann traces on $\partial\Omega_\sigma$ as follows:
\begin{equation}
\begin{aligned}
	h_\sigma &:= h_{\sigma,\sigma'}^\sigma \psi_{\sigma,\sigma'}^\sigma + h_{\sigma,\sigma''}^\sigma \psi_{\sigma,\sigma''}^\sigma, \\
	z_\sigma &:= z_{\sigma,\sigma'}^\sigma \psi_{\sigma,\sigma'}^\sigma + z_{\sigma,\sigma''}^\sigma \psi_{\sigma,\sigma''}^\sigma. 
\end{aligned}
\label{eqn:DefDNtraces}
\end{equation}
Since we have used a smooth partition of unity, the fact that $(h_\sigma,z_\sigma)$ satisfies conditions \eqref{eqn:CompRel:all} is a direct consequence of \eqref{eqn:localregcheck} and \eqref{eqn:localCRcheck}.  Therefore, it follows from \cite[Theorem 1.5.2.8]{grisvard2011elliptic} that there exists a unique $w\in W^{2,p}(\Omega)$ such that $w_\sigma=h_\sigma$ and $\partial_{n_\sigma}w=z_\sigma$ on each $\partial\Omega_\sigma$. Furthermore, we deduce from \eqref{eqn:localHomogBCcheck} that $w$ verifies the boundary conditions of the non-homogeneous Dirichlet-Neumann problem \eqref{eqn:NonHomogDNPb}.

Now, introducing $\tilde{v}:=v-w$ with $v$ the solution to \eqref{eqn:NonHomogDNPb} and $w$ as defined above, we obtain that $\tilde{v}$ solves a Dirichlet-Neumann problem with homogeneous boundary conditions
\begin{equation*}
    \left\{
    \begin{aligned}
        -\Delta \tilde{v} &= k+\Delta w \: \mbox{ in } \Omega , \\
        \tilde{v} &= 0\: \mbox{ on } \partial\Omega_\sigma \mbox{ if } \sigma\in\mathcal{D}, \\
        	\partial_n \tilde{v} &= 0 \: \mbox{ on } \partial\Omega_\sigma \mbox{ if } \sigma\in\mathcal{R}.
    \end{aligned}
    \right.
\end{equation*}
Since $k+\Delta w \in L^p(\Omega)$, it follows from Theorem \ref{thm:RegHomogDNPb} that $\tilde{v}$ exists and is unique in $W^{2,p}(\Omega)$. Thus $v=\tilde{v}+w$ is also in $W^{2,p}(\Omega)$.
\end{proof}

We can now state the regularity result for the case of mixed Dirichlet-Robin non-homogeneous boundary conditions.
\begin{theorem}[Dirichlet-Robin]
	If, in addition to Assumption 1, we assume that $f\in L^p(\Omega)$, $g^{\mathcal{D}}_\sigma\in W^{2-\frac{1}{p},p}(\partial\Omega_\sigma)$ for all $\sigma\in \mathcal{D}$, $p_\sigma\in C^\infty(\overline{\partial\Omega}_\sigma)$ and $g^{\mathcal{R}}_\sigma\in W^{1-\frac{1}{p},p}(\partial\Omega_\sigma)$ for all $\sigma\in \mathcal{R}$, and that we have the compatibility relations
	\begin{itemize}
		\item[$\mathrm{(i)}$] $g^{\mathcal{D}}_\sigma(0)=g^{\mathcal{D}}_{\sigma'}(0)$ if $(\sigma,\sigma')\in\mathcal{C}_{\mathcal{D}}$,
		\item[$\mathrm{(ii)}$] $(g^{\mathcal{D}}_\sigma)'(0)+(p_{\sigma'}g^{\mathcal{D}}_\sigma)(0)=g^{\mathcal{R}}_{\sigma'}(0)$ if $(\sigma,\sigma')\in\mathcal{C}_{\mathcal{M}}$,
	\end{itemize}
	then the solution $u$ to \eqref{eqn:ModelPb} is in $W^{2,p}(\Omega)$.
    \label{thm:RegMixed}
\end{theorem}

\begin{proof}
	The structure of the proof is the same as in the proof of Corollary \ref{cor:RegHomogPb}. First, we will build pairs of functions $(h_\sigma,z_\sigma)$ for all $\sigma\in\mathcal{S}$ satisfying conditions \eqref{eqn:CompRel:all} as well as the boundary conditions prescribed in \eqref{eqn:ModelPb}, that is 
\begin{equation}
	\left\{
    \begin{aligned}
         h_\sigma = g_\sigma^{\mathcal{D}} \: & \mbox{ if } \sigma \in \mathcal{D}, \\
         z_\sigma+p_\sigma h_\sigma = g_\sigma^{\mathcal{R}} \: & \mbox{ if } \sigma \in \mathcal{R}.
    \end{aligned}
    \right.
    \label{eqn:BCcheck}
\end{equation}
This will ensure the existence of an extension $w\in W^{2,p}(\Omega)$ of these boundary conditions. Then, we will be able to conclude using the regularity result for the homogeneous problem from Corollary \ref{cor:RegHomogPb}.

Let $g^{\mathcal{D}}$, $g^{\mathcal{R}}$ and $p$ be such that conditions (i) and (ii) are satisfied. For any $(\sigma,\sigma')\in\mathcal{C}$, we build the local Dirichlet and Neumann traces $(h_{\sigma,\sigma'}^\sigma, z_{\sigma,\sigma'}^\sigma)$ on $\partial\Omega_\sigma$ and $(h_{\sigma,\sigma'}^{\sigma'}, z_{\sigma,\sigma'}^{\sigma'})$ on $\partial\Omega_{\sigma'}$ as follows:
\begin{itemize}
	\item if $(\sigma,\sigma')\in\mathcal{C}_{\mathcal{D}}$, we set
	\begin{equation*}
		h_{\sigma,\sigma'}^\sigma = g_\sigma^{\mathcal{D}}, \quad h_{\sigma,\sigma'}^{\sigma'} = g_{\sigma'}^{\mathcal{D}}, \quad z_{\sigma,\sigma'}^\sigma = (g_{\sigma'}^{\mathcal{D}})', \quad z_{\sigma,\sigma'}^{\sigma'} = (g_\sigma^{\mathcal{D}})';
	\end{equation*}
	\item if $(\sigma,\sigma')\in\mathcal{C}_{\mathcal{R}}$, we set for all $s\in(0,1)$
	\begin{equation*}
	\begin{aligned}
		h_{\sigma,\sigma'}^\sigma &= g_{\sigma'}^{\mathcal{R}}(0)s, \qquad & z_{\sigma,\sigma'}^\sigma &= g_{\sigma}^{\mathcal{R}}-p_\sigma h_{\sigma,\sigma'}^\sigma, \\
		h_{\sigma,\sigma'}^{\sigma'} &= g_{\sigma}^{\mathcal{R}}(0)s, \qquad & z_{\sigma,\sigma'}^{\sigma'} &= g_{\sigma'}^{\mathcal{R}}-p_{\sigma'} h_{\sigma,\sigma'}^{\sigma'};
	\end{aligned}
	\end{equation*}
	\item if $(\sigma,\sigma')\in\mathcal{C}_{\mathcal{M}}$, say $\sigma\in\mathcal{D}$ and $\sigma'\in\mathcal{R}$, we set
	\begin{equation*}
		h_{\sigma,\sigma'}^\sigma = g_\sigma^{\mathcal{D}}, \quad h_{\sigma,\sigma'}^{\sigma'} = g_{\sigma}^{\mathcal{D}}, \quad z_{\sigma,\sigma'}^\sigma = (g_{\sigma}^{\mathcal{D}})', \quad z_{\sigma,\sigma'}^{\sigma'} = g_{\sigma'}^{\mathcal{R}}-p_{\sigma'}h_{\sigma,\sigma'}^{\sigma'}.
	\end{equation*}
\end{itemize}
Using the regularity of $g^{\mathcal{D}}$, $g^{\mathcal{R}}$ and $p$, together with the fact that they satisfy conditions (i) and (ii), we obtain that $(h_{\sigma,\sigma'}^\sigma, z_{\sigma,\sigma'}^\sigma)$ and $(h_{\sigma,\sigma'}^{\sigma'}, z_{\sigma,\sigma'}^{\sigma'})$ verify conditions \eqref{eqn:localregcheck} and \eqref{eqn:localCRcheck}. In addition, we also have
\begin{equation}
	\left\{
	\begin{aligned}
		h_{\sigma,\sigma'}^\sigma &= g_\sigma^{\mathcal{D}} \mbox{ if } \sigma\in\mathcal{D}, \\
		z_{\sigma,\sigma'}^\sigma + p_\sigma h_{\sigma,\sigma'}^\sigma &= g_\sigma^{\mathcal{R}} \mbox{ if } \sigma\in\mathcal{R},
	\end{aligned}
	\right.
	\label{eqn:localBCcheck}
\end{equation}
and the same holds for $(h_{\sigma,\sigma'}^{\sigma'}, z_{\sigma,\sigma'}^{\sigma'})$. Now, on each side $\partial\Omega_\sigma$, we use a partion of unity as in the proof of Corollary \ref{cor:RegHomogPb} and define $(h_\sigma,z_\sigma)$ as in \eqref{eqn:DefDNtraces}. This guarantees that $(h_\sigma,z_\sigma)$ satisfies conditions \eqref{eqn:CompRel:all}, and we deduce existence of a unique $w\in W^{2,p}(\Omega)$ such that $w_\sigma=h_\sigma$ and $\partial_{n_\sigma}w=z_\sigma$ on each $\partial\Omega_\sigma$. In addition, it follows from \eqref{eqn:localBCcheck} that $(h_\sigma,z_\sigma)$ also satisfies \eqref{eqn:BCcheck}, which means that $w$ verifies the boundary conditions of our model problem \eqref{eqn:ModelPb}.

Finally, we introduce $\tilde{u}:=u-w$ where $u$ is the solution to \eqref{eqn:ModelPb} and $w$ is as defined above. A direct computation shows that $\tilde{u}$ solves our model problem with homogeneous boundary conditions, namely
\begin{equation*}
    \left\{
    \begin{aligned}
        -\Delta \tilde{u} &= f+\Delta w \: \mbox{ in } \Omega , \\
        \mathcal{B}\tilde{u} &= 0\: \mbox{ on } \partial\Omega .
    \end{aligned}
    \right.
\end{equation*}
Since $f+\Delta w \in L^p(\Omega)$, existence and uniqueness of $\tilde{u}$ in $W^{2,p}(\Omega)$ follow from Corollary \ref{cor:RegHomogPb}, and we conclude that $u=\tilde{u}+w \in W^{2,p}(\Omega)$.
\end{proof}

\begin{remark}
	Theorem \ref{thm:RegMixed} is a generalization of the result obtained in \cite{grisvard2011elliptic} in the mixed Dirichlet-Neumann non-homogeneous case. Indeed, if we take $p\equiv 0$ on $\partial\Omega$, we recover from (ii) the condition $(g^{\mathcal{D}}_\sigma)'(0)=g^{\mathcal{N}}_{\sigma'}(0)$ at any mixed corner $(\sigma,\sigma')$.
\end{remark}

From now on, we will assume that the data satisfy the regularity required by the assumptions in Theorem \ref{thm:RegDir}, Theorem \ref{thm:RegRob} or Theorem \ref{thm:RegMixed} depending on the definition of the boundary condition operator $\mathcal{B}$. This will guarantee that the solutions we deal with are at least in $W^{2,p}(\Omega)$. In particular, let us recall that since $p>2$ and $\Omega$ is bounded, we have $W^{2,p}(\Omega) \subset C^0(\overline{\Omega})$ and $W^{2-\frac1p,p}(\partial\Omega)\subset W^{1-\frac1p,p}(\partial\Omega)\subset C^0(\partial\Omega)$.

\subsection{Even/odd symmetric decomposition}

We briefly recall the definitions of even/odd symmetric functions in $L^p$, for $p\in [1,+\infty]$.
\begin{defn}[Symmetric set]
    Let $n\geq 1$ be an integer. A subset $U$ of $\mathbb{R}^n$ is
    said to be \textit{symmetric} if, for any $(x_1, \cdots,x_n)\in
    \mathbb{R}^n$,
    $$
      (x_1, \cdots,x_n)\in U \: \Longrightarrow \:
      (-x_1, \cdots,-x_n)\in U.
    $$
\end{defn}
\begin{defn}[Even/odd symmetric functions]
    Let $U\subset \mathbb{R}^n$ be open and symmetric, and
    $p\in[1,+\infty]$. A function $h\in L^p(U)$ is called \textit{even
      symmetric} if for almost all $(x_1, \cdots,x_n)\in U$,
    $$
        h(-x_1, \cdots,-x_n) = h(x_1, \cdots,x_n). 
    $$
    Similarly, the function $h$ is called \textit{odd symmetric} if
    for almost all $(x_1, \cdots,x_n)\in U$,
    $$
        h(-x_1, \cdots,-x_n) = -h(x_1, \cdots,x_n). 
    $$
\end{defn}
Let us also recall the even/odd symmetric decomposition property for any function in $L^p$.
\begin{theorem}[Even/odd decomposition]\label{thm:DecEvenOdd}
  Let $U\subset \mathbb{R}^n$ be open and symmetric, and
  $p\in[1,+\infty]$. Every function $h\in L^p(U)$ can be uniquely
  decomposed as the sum of an even and an odd symmetric function, both in
  $L^p(U)$, which are called the \textit{even symmetric part} and the
  \textit{odd symmetric part} of the function. These functions, denoted by $h_e$
  and $h_o$, are for almost all $(x_1, \cdots,x_n)\in U$ given by
  \begin{eqnarray*}
        h_e(x_1, \cdots,x_n) &:=& \frac{1}{2}\left( h(x_1, \cdots,x_n) + h(-x_1, \cdots,-x_n) \right), \\
        h_o(x_1, \cdots,x_n) &:=& \frac{1}{2}\left( h(x_1, \cdots,x_n) - h(-x_1, \cdots,-x_n) \right).
  \end{eqnarray*}
\end{theorem}
The reader is referred to \cite{chaudet2022cross1} for further details about these notions. 

We can now use this property to decompose problem \eqref{eqn:ModelPb} into its
\textit{even symmetric part} and its \textit{odd symmetric part}. This leads to: find $u^e$ and $u^o$
solutions to
\begin{subequations}\label{subeqn:EvenOddModelPb}
	\begin{align}
		\left\{
    		\begin{aligned}
        		-\Delta u^e &= f_e \: \mbox{ in } \Omega , \\
        		\mathcal{B}u^e &= g_e\: \mbox{ on } \partial\Omega ,
    		\end{aligned}
    		\right.
    		\label{subeqn:EvenModelPb} \\
    		\left\{
    		\begin{aligned}
        		-\Delta u^o &= f_o \: \mbox{ in } \Omega , \\
        		\mathcal{B}u^o &= g_o\: \mbox{ on } \partial\Omega .
    		\end{aligned}
    		\right.
    		\label{subeqn:OddModelPb}
	\end{align}
\end{subequations}
Let $u$ denote the solution to \eqref{eqn:ModelPb}. In the case of the full Dirichlet problem ($\mathcal{D}=\mathcal{S}$), and in the case of the full Robin problem ($\mathcal{R}=\mathcal{S}$) with $p$ even symmetric, it is known from \cite{chaudet2022cross1} that the unique solutions $u^e$ and $u^o$ to these subproblems are precisely the even symmetric part $u_e$ and the odd symmetric part $u_o$ of $u$. 
However, for general boundary conditions $\mathcal{B}u=g$, this identification might not hold. Therefore, in order to perform such a decomposition, we need some additional assumption on the boundary conditions. Let us begin with a definition.

\begin{defn}
	The boundary conditions $\mathcal{B}u=g$ are said to be \textit{symmetric} on $\partial\Omega$ if they satisfy one of the following conditions:
	\begin{itemize}
		\item[$\mathrm{(i)}$] $\mathcal{R}=\emptyset$;
		\item[$\mathrm{(ii)}$] $\mathcal{D}=\emptyset$ and $p$ is even symmetric;
		\item[$\mathrm{(iii)}$] $\{l,r\}\subset \mathcal{R}$ or $\mathcal{D}$, $\{b,t\}\subset \mathcal{R}$ or $\mathcal{D}$, and $p$ is even symmetric.
	\end{itemize}
	\label{dfn:SymBC}
\end{defn}

The last condition (iii) means that boundary conditions on opposite sides of the square $\Omega$ must be of similar nature (either Dirichlet or Robin). This new symmetry assumption enables us to identify the solutions to \eqref{subeqn:EvenModelPb} and \eqref{subeqn:OddModelPb} with the even and odd symmetric parts of $u$.

\begin{lemma}
	If the boundary conditions in \eqref{eqn:ModelPb} are symmetric, then the unique solutions to \eqref{subeqn:EvenModelPb} and \eqref{subeqn:OddModelPb} are given by: $u^e=u_e$ and $u^o=u_o$.
	\label{lem:EvenOddDecSol}
\end{lemma}

\begin{proof}
	The cases of full Dirichlet or Robin problems (conditions (i) or (ii)) have already been treated in \cite{chaudet2022cross1}. Therefore we will focus on the case of mixed boundary conditions (iii). Without loss of generality, let us assume that $\{l,r\}\subset \mathcal{R}$ and $\{b,t\}\subset \mathcal{D}$. In this specific case, problems \eqref{subeqn:EvenOddModelPb} become	
\begin{equation*}
	\left\{
    	\begin{aligned}
    		-\Delta u^e &= f_e \: \mbox{ in } \Omega , \\
    		u^e &= g_{De}\: \mbox{ on } \partial\Omega_b\cup\partial\Omega_t , \\
    		(\partial_n+p)u^e &= g_{Re}\: \mbox{ on } \partial\Omega_l\cup\partial\Omega_r ,
    	\end{aligned}
    	\right.
\end{equation*}
\begin{equation*}
    	\left\{
    	\begin{aligned}
    		-\Delta u^o &= f_o \: \mbox{ in } \Omega , \\
    		u^o &= g_{Do}\: \mbox{ on } \partial\Omega_b\cup\partial\Omega_t , \\
    		(\partial_n+p)u^o &= g_{Ro}\: \mbox{ on } \partial\Omega_l\cup\partial\Omega_r .
    	\end{aligned}
    	\right.  		
\end{equation*}
Since the boundary conditions on opposite sides of $\partial\Omega$ are of similar nature, we are able to define the even/odd symmetric parts of $g_D$ and $g_R$. For example, we have on the bottom side $\partial\Omega_b$
\begin{equation*}
	g_{De}(x,-1) := \frac12 \left( g_D(x,-1)+g_D(-x,1) \right), \: \mbox{ for all } x\in (-1,1),
\end{equation*}
which makes sense because $g_D$ is defined on $\partial\Omega_b\cup\partial\Omega_t$. In addition, we also have in this case that $u^e$ is even symmetric and $u^o$ is odd symmetric. Indeed, introducing $\tilde{u}^e$ and $\tilde{u}^o$ such that for all $(x,y)\in\Omega$, $\tilde{u}^e(x,y):=u^e(-x,-y)$ and $\tilde{u}^o(x,y):=u^o(-x,-y)$, we get that $u^e-\tilde{u}^e$ and $u^o+\tilde{u}^o$ solve the homogeneous mixed problem, whose unique solution is 0. Finally, adding the two previous formulations shows that $u^e+u^o$ solves \eqref{eqn:ModelPb}. Therefore the uniqueness of $u$ together with the uniqueness of the even/odd decomposition in Theorem \ref{thm:DecEvenOdd} yield $u^e=u_e$ and $u^o=u_o$.
\end{proof}

In the rest of this paper, it will be assumed that the boundary conditions in \eqref{eqn:ModelPb} are symmetric.

% -----------------------------------------------------
\section{A new variant of the Dirichlet-Neumann method}
\label{sec:NewDN}
% -----------------------------------------------------

The analysis of the (standard) Dirichlet-Neumann method performed in the first part of this work
\cite[Section 3]{chaudet2022cross1} reveals that the method is extremely efficient when applied to the even symmetric part of \eqref{eqn:ModelPb}, while it completely fails when applied to its odd symmetric part. Here, the idea is to take advantage of the nice behaviour of the method for the even symmetric part of the problem, and designing another method of Dirichlet-Neumann type specifically tuned to treat the odd symmetric part.

\subsection{Presentation of the method}

As in \cite[Section 1.4]{quarteroni1999domain}, we introduce a gray and white
coloring, see Figure \ref{sub:SchDDstd}, and define the sets of indices
$\mathcal{I}_G:=\{1\leq i\leq 4 \: : \: \Omega_i \mbox{ is gray }\} =
\{2,4\}$ and $\mathcal{I}_W:=\mathcal{I}\setminus \mathcal{I}_G =
\{1,3\}$. The transmission conditions of the standard Dirichlet-Neumann method are recalled in Figure \ref{sub:SchDDstd}.
Similarly to the approach in \cite{chaudet2023cross}, we propose a new distribution of Dirichlet and Neumann transmission conditions for the odd symmetric part, as shown in Figure \ref{sub:SchDDnew}. Let us introduce $\Gamma_D$ and $\Gamma_N$ the sets containing all parts of the interface $\Gamma$ where transmission conditions of Dirichlet or Neumann type are imposed, that is:
\begin{equation*}
\begin{aligned}
	\Gamma_D & := \{ \Gamma_{ij} \mid (i,j)\in\mathcal{I}^2 \mbox{ and } \Gamma_{ij} \mbox{ is of Dirichlet type} \} \\
	\:       &  = \{ \Gamma_{14}, \Gamma_{21}, \Gamma_{32}, \Gamma_{43} \}\:, \\
	\Gamma_N & := \{ \Gamma_{ij} \mid (i,j)\in\mathcal{I}^2 \mbox{ and } \Gamma_{ij} \mbox{ is of Neumann type} \} \\
	\:       &  = \{ \Gamma_{12}, \Gamma_{23}, \Gamma_{34}, \Gamma_{41} \}\:.
\end{aligned}
\end{equation*}
Given an initial guess $u^0=u_e^0+u_o^0$ and a relaxation parameter
$\theta\in \mathbb{R}$, each iteration $k\geq 1$ of our new Dirichlet-Neumann method applied to \eqref{eqn:ModelPb} can be split into two steps. In each step, we solve separately and use different transmission conditions for the even and odd symmetric parts, which leads to the following algorithm.

\begin{itemize}
    \item  \textbf{(First step)} For the even symmetric part, solve for all $i\in \mathcal{I}_W$
    \begin{equation*}
        \left\{  \begin{aligned}
            -\Delta u_{e,i}^k  &= f_e \: \mbox{ in } \Omega_i \:, \\
            \mathcal{B}u_{e,i}^k &= g_e \: \mbox{ on } \partial\Omega_i^0 \:, \\
            u_{e,i}^k &= \theta u_{e,j}^{k-1} + (1-\theta)u_{e,i}^{k-1}\: \mbox{ on } \Gamma_{ij}, \: \forall j\in \mathcal{I}_G \: \mbox{ s.t. } \Gamma_{ij}\neq\emptyset \:.
        \end{aligned}
        \right.
    \end{equation*}
    For the odd symmetric part, solve for all $i\in \mathcal{I}_W$
    \begin{equation*}
        \left\{  \begin{aligned}
            -\Delta u_{o,i}^k  &= f_o \: \mbox{ in } \Omega_i \:, \\
            \mathcal{B}u_{o,i}^k &= g_o \: \mbox{ on } \partial\Omega_i^0 \:, \\
            u_{o,i}^k &= \theta u_{o,j}^{k-1} + (1-\theta)u_{o,i}^{k-1}\: \mbox{ on } \Gamma_{ij}, \: \forall j\in \mathcal{I} \: \mbox{ s.t. } \Gamma_{ij}\in\Gamma_D \:, \\
            \partial_{n_i}u_{o,i}^k &= -\theta\partial_{n_j}u_{o,j}^{k-1} + (1-\theta)\partial_{n_i}u_{o,i}^{k-1} \: \mbox{ on } \Gamma_{ij}, \: \forall j\in \mathcal{I} \: \mbox{ s.t. } \Gamma_{ij}\in\Gamma_N \:.
        \end{aligned}
        \right.
    \end{equation*}
        \item  \textbf{(Second step)} For the even symmetric part, solve for all $i\in \mathcal{I}_G$
    \begin{equation*}
        \left\{  \begin{aligned}
            -\Delta u_{e,i}^k  &= f_e \: \mbox{ in } \Omega_i \:, \\
            \mathcal{B}u_{e,i}^k &= g_e \: \mbox{ on } \partial\Omega_i^0 \:, \\
            \partial_{n_i}u_{e,i}^k &= -\partial_{n_j}u_{e,j}^{k} \: \mbox{ on } \Gamma_{ij}, \: \forall j\in \mathcal{I}_G \: \mbox{ s.t. } \Gamma_{ij}\neq\emptyset \:.
        \end{aligned}
        \right.
    \end{equation*}
    For the odd symmetric part, solve for all $i\in \mathcal{I}_G$
    \begin{equation*}
        \left\{  \begin{aligned}
            -\Delta u_{o,i}^k  &= f_o \: \mbox{ in } \Omega_i \:, \\
            \mathcal{B}u_{o,i}^k &= g_o \: \mbox{ on } \partial\Omega_i^0 \:, \\
            u_{o,i}^k &= u_{o,j}^{k} \: \mbox{ on } \Gamma_{ij}, \: \forall j\in \mathcal{I} \: \mbox{ s.t. } \Gamma_{ij}\in\Gamma_D \:, \\
            \partial_{n_i}u_{o,i}^k &= -\partial_{n_j}u_{o,j}^{k}\: \mbox{ on } \Gamma_{ij}, \: \forall j\in \mathcal{I} \: \mbox{ s.t. } \Gamma_{ij}\in\Gamma_N \:.
        \end{aligned}
        \right.
    \end{equation*}
\end{itemize}
As for the standard Dirichlet-Neumann method, we have to choose an initial guess $u^0$. For the method to be well defined, this initial guess is required to satisfy the following compatibility condition.
\begin{defn}[Compatible initial guess]
  An initial guess $u^0$ is said to be
  compatible with the boundary conditions if it satisfies: $u^0\in W^{2,p}(\Omega)$, 
  $u^0\vert_{\partial\Omega_\sigma\cap\Gamma} = g^{\mathcal{D}}_\sigma\vert_\Gamma$ and $\partial_\tau u^0\vert_{\partial\Omega_\sigma\cap\Gamma} = (g^{\mathcal{D}}_{\sigma})' \vert_{\Gamma}$ for all $\sigma\in\mathcal{D}$, and $(\partial_n+p)u^0\vert_{\partial\Omega_\sigma\cap\Gamma} = g^{\mathcal{R}}_{\sigma}\vert_\Gamma$ for all $\sigma\in\mathcal{R}$, where $\tau$ denotes the tangential vector to the boundary $\partial\Omega$.
  \label{def:CompatibleIG}
\end{defn}
This definition means that, in each subdomain, $u^0$ and the boundary data must satisfy the regularity assumptions and the $W^{2,p}$ compatibility relations (i)-(ii) in Theorem \ref{thm:RegMixed}.

\subsection{Well-posedness}

One issue when using the standard Dirichlet-Neumann method in configurations involving cross-points is well-posedness. More specifically, it can be proved that the method is in general not well-posed as the solutions in the subdomains $\Omega_i$, $i\in\mathcal{I}$, might not be unique and belong to $L^2(\Omega_i)\setminus H^1(\Omega_i)$, see \cite{chaudet2022cross1}. The new Dirichlet-Neumann method proposed here does not suffer from this non desirable behaviour. Instead, when the initial guess is regular enough, we are able to prove that all the local solutions computed along the iterations remain in $W^{2,p}(\Omega_i)$, for each $i\in\mathcal{I}$.

\begin{theorem}
	Let $u^0$ be an initial guess compatible with the boundary conditions. Then the new Dirichlet-Neumann method is well-posed. In addition, for all $k\geq 1$ and for each $i\in\mathcal{I}$, the local solution $u_i^k$ belongs to $W^{2,p}(\Omega_i)$.
	\label{thm:WellPosed}
\end{theorem}

\begin{proof}
	We will begin with studying the first two iterations, then we will conclude by induction.
	
\noindent \textbf{Iteration $\boldsymbol{k=1}$, first step:} In
$\Omega_1$, we solve for the even and odd symmetric parts
\begin{equation*}
	\left\{\begin{aligned}
    		-\Delta u_{e,1}^1 &= f_e \: \mbox{ in } \Omega_1 \:, \\
    		\mathcal{B} u_{e,1}^1 &= g_e \: \mbox{ on } \partial\Omega_1^0 \:, \\
    		u_{e,1}^1 &= u^0_e  \: \mbox{ on } \Gamma_{14}\cup\Gamma_{12} \:,
  	\end{aligned}\right.
  	\quad \mbox{ and } \quad
  	\left\{\begin{aligned}
    		-\Delta u_{o,1}^1 &= f_o \: \mbox{ in } \Omega_1 \:, \\
   		 \mathcal{B} u_{o,1}^1 &= g_o \: \mbox{ on } \partial\Omega_1^0 \:, \\
   		 u_{o,1}^1 &= u^0_o  \: \mbox{ on } \Gamma_{14} \:, \\
    		\partial_{n_1} u_{o,1}^1 &= \partial_{n_1}u^0_o \: \mbox{ on } \Gamma_{12} \:.
  	\end{aligned}\right.
\end{equation*}
Since we are dealing with Dirichlet-Robin problems for both parts, we need to check that the functions in the boundary conditions are regular enough and that the compatibility relations (i)-(ii) in Theorem \ref{thm:RegMixed} are satisfied on $\partial\Omega_1$. In the two cases, the regularity condition follows directly from the regularity assumptions on the data and on $u^0$. As for conditions (i) and (ii), they hold at the bottom-left and top-right corners. Indeed, on the one hand, the boundary data related to the whole domain are already assumed to satisfy the compatibility relations (i)-(ii), and on the other hand, the restriction of $u^0$ to $\Omega_1$ belongs to $W^{2,p}(\Omega_1)$ so conditions \eqref{eqn:CompRel:all} ensure that (i)-(ii) are verified.
It remains to prove that (i) and (ii) hold at the corners corresponding to boundary cross-points.
Each of these corners corresponds to an intersection $\partial\Omega_\sigma \cap \Gamma$, with $\sigma\in\{l,b\}$. Let us start with the even symmetric part. In the case we only enforce Dirichlet transmission conditions on $\Gamma_{12}$ and $\Gamma_{41}$, the type of the corner only depends on $\sigma$. If $\sigma\in\mathcal{D}$, then the corner is of Dirichlet type and we already have $u^0_e = g^\mathcal{D}_e$ at this vertex since $u^0$ is compatible with the boundary conditions. If $\sigma\in\mathcal{R}$, then the corner is mixed and, taking the even symmetric part of the condition $(\partial_n+p)u^0\vert_{\partial\Omega_\sigma\cap\Gamma} = g^{\mathcal{R}}_{\sigma}\vert_\Gamma$, we obtain that (ii) is satisfied. Now, we turn to the odd symmetric case, where a Dirichlet boundary condition is enforced on $\Gamma_{41}$ and a Neumann boundary condition is enforced on $\Gamma_{12}$. Let us first focus on the left top corner $\Gamma\cap\partial\Omega_l=\Gamma_{41}\cap\partial\Omega_l$. If $l\in\mathcal{R}$, then condition (ii) must be satisfied, and this is indeed guaranteed by taking the odd symmetric part of the equality $(\partial_n+p)u^0\vert_{\partial\Omega_l\cap\Gamma} = g^{\mathcal{R}}_{l}\vert_\Gamma$. Else, if $l\in\mathcal{D}$, then we already have the continuity condition $u_o^0=g_o^{\mathcal{D}}$ at the corner. As for the bottom right corner of $\Omega_1$, i.e. $\Gamma\cap\partial\Omega_b=\Gamma_{12}\cap\partial\Omega_b$, we need again to distinguish between the cases $b\in\mathcal{R}$ or $\mathcal{D}$. If $b\in\mathcal{R}$, then we have a Robin corner and no additional condition is required for $W^{2,p}$ regularity. Else, if $b\in\mathcal{D}$, then the corner is mixed and condition (ii) is obtained by taking the odd symmetric part of $\partial_\tau u^0\vert_{\partial\Omega_b\cap\Gamma} = (g^{\mathcal{D}}_b)' \vert_{\Gamma}$.
Finally, it follows from Theoreom \ref{thm:RegMixed} that the even and odd symmetric parts of $u_1^1$ both belong to $W^{2,p}(\Omega_1)$, which yields $u_1^1\in W^{2,p}(\Omega_1)$ as well. By symmetry, we deduce that the even and odd symmetric parts of $u_3^1$ also belong to $W^{2,p}(\Omega_3)$, and thus $u_3^1\in W^{2,p}(\Omega_3)$.

\noindent \textbf{Iteration $\boldsymbol{k=1}$, second step:} Now, in
$\Omega_2$, we solve for the even and odd symmetric parts
\begin{equation*}
	\left\{\begin{aligned}
    		-\Delta u_{e,2}^1 &= f_e \: \mbox{ in } \Omega_2 \:, \\
    		\mathcal{B} u_{e,2}^1 &= g_e \: \mbox{ on } \partial\Omega_2^0 \:, \\
    		\partial_{n_2}u_{e,2}^1 &= -\partial_{n_1}u_{e,1}^1  \: \mbox{ on } \Gamma_{12}\:, \\
    		\partial_{n_2}u_{e,2}^1 &= -\partial_{n_3}u_{e,3}^1  \: \mbox{ on } \Gamma_{23}\:,
  	\end{aligned}\right.
  	\quad \mbox{ and } \quad
  	\left\{\begin{aligned}
    		-\Delta u_{o,2}^1 &= f_o \: \mbox{ in } \Omega_2 \:, \\
   		 \mathcal{B} u_{o,2}^1 &= g_o \: \mbox{ on } \partial\Omega_2^0 \:, \\
   		 u_{o,2}^1 &= u_{o,1}^1  \: \mbox{ on } \Gamma_{12} \:, \\
    		\partial_{n_2} u_{o,2}^1 &= -\partial_{n_3}u_{o,3}^1 \: \mbox{ on } \Gamma_{23} \:.
  	\end{aligned}\right.
\end{equation*}
Here again, at the corner which is not contained in $\Gamma$ (bottom-right corner in $\Omega_2$), conditions (i)-(ii) follow from the assumptions on the boundary data. Let us now study the three other corners. In the even symmetric case, the top-left corner is of Robin type so no specific compatibility condition is required. As for boundary cross-points, since $u_1^1\in W^{2,p}(\Omega_1)$ $u_3^1\in W^{2,p}(\Omega_3)$, we know that $u_{e,1}^1$ and $u_{e,3}^1$ satisfy conditions \eqref{eqn:CompRel:all}. Thus, no matter the type of boundary condition enforced on $\partial\Omega_\sigma$ for $\sigma\in \{b,r\}$, conditions (i)-(ii) are necessarily satisfied in $\Omega_2$ as $g^{\mathcal{D}}$ and $g^{\mathcal{R}}$ are continuous at the boundary cross-points.
In the odd symmetric case, the top-left corner is mixed. Using the symmetry relation $u_{o,3}^1(x,y)=-u_{o,1}^1(-x,-y)$ for all $(x,y)\in\overline{\Omega}_3$ and the fact that $u_{o,1}^1$ satisfies \eqref{eqn:CompRel:3}, we obtain that condition (ii) holds at this corner. As for boundary cross-points, we have that $u_{o,1}^1$ and $u_{o,3}^1$ satisfy \eqref{eqn:CompRel:all} so we can use the same argument as in the even symmetric case. Consequently, we conclude that the compatibility relations (i)-(ii) hold in both cases, which means that $u_{e,2}^1$ and $u_{o,2}^1$ are in $W^{2,p}(\Omega_2)$, and therefore $u_2^1\in W^{2,p}(\Omega_2)$. Using symmetry arguments, we also get that $u_4^1\in W^{2,p}(\Omega_4)$.

\noindent \textbf{Iteration $\boldsymbol{k=2}$, first step:}
In $\Omega_1$, we solve for the even symmetric part
\begin{equation*}
	\left\{\begin{aligned}
    		-\Delta u_{e,1}^2 &= f_e \: \mbox{ in } \Omega_1 \:, \\
    		\mathcal{B} u_{e,1}^2 &= g_e \: \mbox{ on } \partial\Omega_1^0 \:, \\
    		u_{e,1}^2 &= \theta u_{e,2}^1 + (1-\theta) u_{e,1}^1 \: \mbox{ on } \Gamma_{12} \:, \\
    		u_{e,1}^2 &= \theta u_{e,4}^1 + (1-\theta) u_{e,1}^1 \: \mbox{ on } \Gamma_{41} \:.
  	\end{aligned}
  	\right.
\end{equation*}
Since all local solutions at iteration $k=1$ have $W^{2,p}$ regularity, the boundary data in the previous system has the required regularity. It remains to check that the compatibility relations are statisfied at each corner of $\Omega_1$.
Using the symmetry relation $u_{e,4}^1(x,y)=u_{e,2}^1(-x,-y)$ for all $(x,y)\in\overline{\Omega}_4$, we get that $u_{e,4}^1(0,0)=u_{e,2}^1(0,0)$ and thus condition (i) is verified at the top-right corner. The appropriate condition also holds at the bottom-left corner since the boundary data of problem \eqref{eqn:ModelPb} are assumed to satisfy (i)-(ii). As for boundary cross-points, let us start with the bottom-right corner $\Gamma_{12}\cap\partial\Omega_b$. No matter the boundary condition enforced on $\partial\Omega_b$, we know that $u_{e,1}^1$ and $u_{e,2}^1$ satisfy the proper compatibility relation since they both have $W^{2,p}$ regularity. Therefore the convex combination $\theta u_{e,2}^1+(1-\theta)u_{e,1}^1$ satisfies the same condition. Applying the same reasoning for the top-left corner, we conclude that all the assumptions of Theorem \ref{thm:RegMixed} are verified, hence $u_{e,1}^2\in W^{2,p}(\Omega_1)$. By symmetry, we also have a similar result in $\Omega_3$, that is $u_{e,3}^2\in W^{2,p}(\Omega_3)$. \\
Let us now turn to the odd symmetric case. In $\Omega_1$, the odd symmetric part of the local solution solves
\begin{equation*}  	
  	\left\{\begin{aligned}
    		-\Delta u_{o,1}^2 &= f_o \: \mbox{ in } \Omega_1 \:, \\
   		 \mathcal{B} u_{o,1}^2 &= g_o \: \mbox{ on } \partial\Omega_1^0 \:, \\
   		 u_{o,1}^2 &= \theta u_{o,4}^1 + (1-\theta) u_{o,1}^1  \: \mbox{ on } \Gamma_{41} \:, \\
    		\partial_{n_1} u_{o,1}^2 &= -\theta \partial_{n_2}u_{o,2}^1 + (1-\theta) \partial_{n_1}u_{o,1}^1 \: \mbox{ on } \Gamma_{12} \:.
  	\end{aligned}\right.
\end{equation*}
Using the same arguments as in the even symmetric case, we can prove that the regularity assumption on the boundary data as well as the compatibility relations at the bottom-left, bottom-right and top-left corners are satisfied. As for the top-right corner, since $u_{o,1}^1\in W^{2,p}(\Omega_1)$, we get that the pair $(u_{o,1}^1,\partial_{n_1}u_{o,1}^1)$ satisfies (ii). As for the pair $(u_{o,4}^1,-\partial_{n_2}u_{o,2}^1)$, we have at this corner
\begin{equation*}
	\begin{aligned}
		(u_{o,4}^1)'(0) &= \partial_x ( u_{o,4}^1(x,0) )\vert_{x=0} = -\partial_x ( u_{o,2}^1(-x,0) ) \vert_{x=0} \\
		&= \partial_x u_{o,2}^1 (0,0) = -(\partial_{n_2}u_{o,2}^1)(0)\:,
	\end{aligned}
\end{equation*}
which means that this pair also satisfies (ii). Hence, the convex combination of these two pairs still satisfies (ii) and it follows that $u_{o,1}^2\in W^{2,p}(\Omega_1)$. By symmetry, we also get that $u_{o,3}^2\in W^{2,p}(\Omega_3)$. Gathering the results for the even and odd symmetric cases, we end up with $u_{1}^2\in W^{2,p}(\Omega_1)$ and $u_{3}^2\in W^{2,p}(\Omega_3)$.

\noindent \textbf{Iteration $\boldsymbol{k=2}$, second step:}
Since there is no convex combination in the transmission conditions here, the proof is exactly the same as the one for the second step of iteration $k=1$.

\noindent \textbf{Iteration $\boldsymbol{k\geq 3}$:}
The rest of the proof follows by induction. Indeed, proving the result for any $k\geq 3$ assuming it holds for $k-1$ can be done by following the exact same steps as in proving the result for $k=2$.
\end{proof}

Due to this result, we know that under the appropriate assumptions on the initial guess $u^0$, in each subdomain $\Omega_i$, the local solution obtained at any iteration $k$ is continuous on $\overline{\Omega}_i$, and both its trace and normal derivative are continuous on $\partial\Omega_i$. 

\begin{remark}
	Even if the conditions on $u^0$ appear at first sight difficult to realize in practice, it is actually quite easy to build such an initial guess in a systematic way using piecewise polynomial functions. More specifically, let us denote each boundary cross-point by $\mathrm{x}_i := \Gamma_{i,i+1}\cap\partial\Omega$ for each $i\in\mathcal{I}$. The conditions on $u^0$ mean that at each $\mathrm{x}_i$, we must have
\begin{equation}
	\left\{
	\begin{aligned}
		u^0(\mathrm{x}_i) &= \alpha_i, \\
		\partial_n u^0(\mathrm{x}_i) &= \beta_i, \\
		\partial_\tau u^0(\mathrm{x}_i) &= \gamma_i,
	\end{aligned}	
	\right.
	\label{eqn:Condu0}
\end{equation} 
	for some given $\alpha_i$, $\beta_i$ and $\gamma_i$ depending on the data. First, in order to satisfy the first two conditions in \eqref{eqn:Condu0}, we can build for each $i\in\mathcal{I}$ a piecewise polynomial function $p_i\in W^{2,p}(\Omega)$ such that
\begin{equation*}
	\left\{
	\begin{aligned}
		p_i(\mathrm{x}_i) &= \alpha_i, \\
		\partial_n p_i(\mathrm{x}_i) &= \beta_i, \\
		\partial_\tau p_i(\mathrm{x}_i) &= 0,
	\end{aligned}	
	\right.
	\quad \mbox{ and } \quad
	\left\{
	\begin{aligned}
		p_i(\mathrm{x}_j) &= 0, \\
		\partial_n p_i(\mathrm{x}_j) &= 0, \\
		\partial_\tau p_i(\mathrm{x}_j) &= 0,
	\end{aligned}	
	\right.	
\end{equation*} 
	for each $j\neq i$. 
	Then, in order to satisfy the third condition in \eqref{eqn:Condu0}, we build another piecewise polynomial function $q_i\in W^{2,p}(\Omega)$ such that 
\begin{equation*}
	\left\{
	\begin{aligned}
		q_i(\mathrm{x}_i) &= 0, \\
		\partial_n q_i(\mathrm{x}_i) &= 0, \\
		\partial_\tau q_i(\mathrm{x}_i) &= \gamma_i,
	\end{aligned}	
	\right.
	\quad \mbox{ and } \quad
	\left\{
	\begin{aligned}
		q_i(\mathrm{x}_j) &= 0, \\
		\partial_n q_i(\mathrm{x}_j) &= 0, \\
		\partial_\tau q_i(\mathrm{x}_j) &= 0,
	\end{aligned}	
	\right.	
\end{equation*}
	for $j\neq i$. This ensures that the choice $u^0:=\sum_{i\in\mathcal{I}}(p_i+q_i)$ satisfies all conditions in \eqref{eqn:Condu0}.
	Taking for example $i=2$, we can set
\begin{equation*}
	p_2(x,y) = 
	\left\{
	\begin{aligned}
		& a_2x^3+b_2x^2 & \mbox{in } \Omega_2\cup\Omega_3, \\
		& 0 & \mbox{in } \Omega_1\cup\Omega_4,
	\end{aligned}	
	\right.
	\quad \mbox{ with }
	\left\{
	\begin{aligned}
		a_2+b_2 &= \alpha_2, \\
		3a_2+2b_2 &= \beta_2,
	\end{aligned}	
	\right.	
\end{equation*}
\begin{equation*}
	q_2(x,y) = 
	\left\{
	\begin{aligned}
		& \gamma_2 xy^2 & \mbox{in } \Omega_2\cup\Omega_3, \\
		& 0 & \mbox{in } \Omega_1\cup\Omega_4.
	\end{aligned}	
	\right.	
\end{equation*}
	And we build the other $p_i$, $q_i$ in a similar fashion. This provides us with an analytical technique to construct a suitable initial guess for any boundary data.
\end{remark}

\subsection{Convergence analysis}

We already know that our new method is very efficient for dealing with the even symmetric part. This is a direct consequence of \cite[Theorem 6 and Theorem 13]{chaudet2022cross1}, which we recall here. In what follows, let $u^0$ be an initial guess such that $u^0$ is compatible with the boundary conditions.
 
\begin{theorem}
  Taking $u^0_e$ as the initial guess for the new Dirichlet-Neumann
  method applied to the even symmetric part of problem
  \eqref{eqn:ModelPb} produces a sequence $\{u_e^k\}_k$ that
  converges geometrically to the solution $u_e$ in the $L^2$-norm and
  the broken $H^1$-norm, for any $\theta\in (0,1)$. Moreover, the
  convergence factor is given by $\vert1-2\theta\vert$, which also proves that this method becomes a direct solver for the
  specific choice $\theta=\frac{1}{2}$.
    \label{thm:GeoCvgEvenDirDN}
\end{theorem}

\begin{proof}
	This result is slightly more general than the ones in \cite{chaudet2022cross1}. Indeed, only the full Dirichlet and Robin cases (i.e. $\mathcal{R}=\emptyset$ or $\mathcal{D}=\emptyset$) were treated, while we cover here all types of symmetric boundary conditions. Since we have seen in the proof of Lemma \ref{lem:EvenOddDecSol} that the notions of even/odd symmetry were preserved when the boundary conditions are symmetric, all the arguments in the proof of \cite[Theorem 6]{chaudet2022cross1} based on symmetry properties of the local errors remain valid here. Thus, following the exact same steps, we end up with a similar estimate for the recombined error in the $L^2$ and broken $H^1$ norms.
\end{proof}

Now, given the choice of transmission conditions for our new Dirichlet-Neumann method, we are able to prove the same result for the odd symmetric part.
\begin{theorem}
  Taking $u^0_o$ as the initial guess for the new Dirichlet-Neumann
  method applied to the odd symmetric part of problem
  \eqref{eqn:ModelPb} produces a sequence $\{u_o^k\}_k$ that
  converges geometrically to the solution $u_o$ in the $L^2$-norm and
  the broken $H^1$-norm, for any $\theta\in (0,1)$. Moreover, the
  convergence factor is given by $\vert1-2\theta\vert$, which also proves that this method becomes a direct solver for the
  specific choice $\theta=\frac{1}{2}$.
    \label{thm:GeoCvgOddDirDN}
\end{theorem}
\begin{proof}
We proceed as in the proof of \cite[Theorem 6]{chaudet2022cross1}. First, we rewrite the problem in terms of the local errors $e_i^k:=u_i-u_i^k$ where $u_i:=u\vert_{\Omega_i}$ is the restriction of the original solution to the $i-$th subdomain. Then we perform the first iterations of the method.

\noindent \textbf{Iteration $\boldsymbol{k=1}$, first step:} In
$\Omega_1$, the local error satisfies
\begin{equation*}
  \left\{\begin{aligned}
    -\Delta e_{o,1}^1 &= 0 \: \mbox{ in } \Omega_1 \:, \\
    \mathcal{B} e_{o,1}^1 &= 0 \: \mbox{ on } \partial\Omega_1^0 \:, \\
    e_{o,1}^1 &= u_o-u^0_o  \: \mbox{ on } \Gamma_{14} \:, \\
    \partial_{n_1} e_{o,1}^1 &= \partial_{n_1}(u_o-u^0_o) \: \mbox{ on } \Gamma_{12} \:.
  \end{aligned}\right.
\end{equation*}
Since $u^0_o$ is compatible with the odd part of the Dirichlet
boundary condition, we know from Theorem \ref{thm:WellPosed} that $e_{o,1}^1$ exists and is unique in $W^{2,p}(\Omega_1)$. \\
In the same way, the local error satisfies in $\Omega_3$
\begin{equation*}
  \left\{\begin{aligned}
    -\Delta e_{o,3}^1 &= 0 \: \mbox{ in } \Omega_3 \:, \\
    \mathcal{B} e_{o,3}^1 &= 0 \: \mbox{ on } \partial\Omega_1^0 \:, \\
    e_{o,3}^1 &= u_o-u^0_o  \: \mbox{ on } \Gamma_{32} \:, \\
    \partial_{n_3} e_{o,3}^1 &= \partial_{n_3}(u_o-u^0_o) \: \mbox{ on } \Gamma_{34} \:.
  \end{aligned}\right.
\end{equation*}
By the same arguments as before, this problem is well-posed in $W^{2,p}(\Omega_3)$. Since $u_o-u^0_o$ is odd symmetric, it follows that the only
solution $e_{o,3}^1$ to this problem verifies
$e_{o,3}^1(x,y) = -e_{o,1}^1(-x,-y)$, for all $(x,y)\in
\overline{\Omega}_3$.

\noindent \textbf{Iteration $\boldsymbol{k=1}$, second step:} Now, in
$\Omega_2$, the local error satisfies
\begin{equation*}
  \left\{  \begin{aligned}
    -\Delta e_{o,2}^1 &= 0 \: \mbox{ in } \Omega_2 \:, \\
    \mathcal{B} e_{o,2}^1 &= 0 \: \mbox{ on } \partial\Omega_2^0 \:, \\
    e_{o,2}^1 &= e_{o,1}^1 = (e_{o,1}^1(-x,y))\vert_{x=0}  \: \mbox{ on } \Gamma_{21} \:,\\
   \partial_y e_{o,2}^1 &= \partial_y e_{o,3}^1 = \partial_y (-e_{o,1}^1(-x,-y))\vert_{y=0} = \partial_y (e_{o,1}^1(-x,y))\vert_{y=0} \: \mbox{ on } \Gamma_{23} \:.
  \end{aligned}\right.
\end{equation*}
This problem admits a unique solution in $W^{2,p}(\Omega_2)$. Moreover, the
function defined in $\Omega_2$ by $(x,y)\mapsto e_{o,1}^1(-x,y)$
solves the problem. Therefore, by uniqueness, we get that $e_{o,2}^1(x,y) = e_{o,1}^1(-x,y)$, for all
$(x,y)\in\overline{\Omega}_2$. Again, in the exact same way, we get that the
solution $e_{o,4}^1$ in $\Omega_4$ is given by $e_{o,4}^1(x,y) =
-e_{o,1}^1(x,-y)$, for all $(x,y)\in\overline{\Omega}_4$. A numerical illustration of these symmetry properties can be found in Figure \ref{sub:Err1Th045Ex2}.

Note that, due to the odd symmetry of $u_o$ and $u_o^0$, one has $e_{o,1}^1(0,0)=(u_o-u_o^0)(0,0)=0$. Therefore the recombined error $e_o^1$ (defined in $\Omega\setminus\Gamma$) obtained after the first iteration is continuous at the cross-point. Actually, given its symmetry properties, it is also continuous on $\Gamma_{12}\cup\Gamma_{34}$. However, it might be discontinuous across the rest of the skeleton $\Gamma_{23}\cup\Gamma_{41}$. This may lead to
discontinuities for the recombined solution $u_o^1=u_o+e_o^1$. In addition, $e_o^1$ is odd symmetric in $\Omega\setminus\Gamma$.
        
\noindent\textbf{Iteration $\boldsymbol{k=2}$, first step:} In
$\Omega_1$, the local error satisfies
\begin{equation*}
   \left\{  \begin{aligned}
     -\Delta e_{o,1}^2 &= 0 \: \mbox{ in } \Omega_1 \:, \\
     \mathcal{B} e_{o,1}^2 &= 0 \: \mbox{ on } \partial\Omega_1^0 \:, \\
     e_{o,1}^2 &= \theta e_{o,4}^1+(1-\theta)e_{o,1}^1 = (1-2\theta)e_{o,1}^1  \: \mbox{ on } \Gamma_{14} \:,\\
     \partial_x e_{o,1}^2 &= \theta \partial_x e_{o,2}^1+(1-\theta)\partial_x e_{o,1}^1 = (1-2\theta) \partial_x e_{o,1}^1 \: \mbox{ on } \Gamma_{12} \:.
   \end{aligned}\right.
\end{equation*}
The unique solution to this problem is $e_{o,1}^2 =
(1-2\theta)e_{o,1}^1$.  Similarly, we obtain in $\Omega_3$,
$e_{o,3}^2 = (1-2\theta)e_{o,3}^1$.
      
\noindent\textbf{Iteration $\boldsymbol{k=2}$, second step:} Now, in
$\Omega_2$, the local error satisfies
\begin{equation*}
   \left\{  \begin{aligned}
         -\Delta e_{o,2}^2 &= 0 \: \mbox{ in } \Omega_2 \:, \\
         \mathcal{B} e_{o,2}^2 &= 0 \: \mbox{ on } \partial\Omega_2^0 \:, \\
         e_{o,2}^2 &= e_{o,1}^2 = (e_{o,1}^2(-x,y))\vert_{x=0}  \: \mbox{ on } \Gamma_{12} \:,\\
         \partial_y e_{o,2}^2 &= \partial_y e_{o,3}^2 = \partial_y (e_{o,1}^2(-x,y))\vert_{y=0} \: \mbox{ on } \Gamma_{23} \:.
         \end{aligned}
         \right.
\end{equation*}
The unique solution to this problem is given by $e_{o,2}^2(x,y) = e_{o,1}^2(-x,y) =
(1-2\theta)e_{o,2}^1(x,y)$ for all $(x,y)\in\overline{\Omega}_2$.  Again, in the exact same way, we get in $\Omega_4$, $e_{o,4}^2 = (1-2\theta)e_{o,4}^1$. Therefore, for all
$\theta\in(0,1)$, the recombined error solution $e_o^2$ is exactly
$(1-2\theta)e_o^1$.

\noindent \textbf{Iterations $\boldsymbol{k\geq 3}$:} Given the behaviour of the method for the first two iterations, it follows by induction that, at
iteration $k$, one has
\begin{equation*}
    e_o^k = (1-2\theta)^{k-1} e_o^1 \hspace{1em} \mbox{ in } \Omega\setminus\Gamma\:,
\end{equation*}
see Figure \ref{sub:Err7Th045Ex1}.     
This proves that our new Dirichlet-Neumann method converges geometrically to the solution
$u_o$ both in the $L^2$-norm and the broken $H^1$-norm for all
$\theta\in(0,1)$,
\begin{equation*}
        \parallel u_o - u_o^k\parallel_{L^2(\Omega)} \: \leq \: C \vert1-2\theta\vert^{k-1} \hspace{0.5em} \mbox{and} \hspace{0.5em} \sum_{i\in \mathcal{I}} \parallel u_{o,i} - u_{o,i}^k\parallel_{H^1(\Omega_i)} \: \leq \: C' \vert1-2\theta\vert^{k-1} \:.
\end{equation*}
$\:$
\end{proof}

Now that we have established the good properties of our new method when dealing with both the even and odd symmetric parts of problem \eqref{eqn:ModelPb}, we can state our main result, which follows directly from the previous theorems.

\begin{cor}
	For any $u^0$ compatible with the boundary condition, our new Dirichlet-Neumann method applied to \eqref{eqn:ModelPb} produces a sequence $\{u^k\}_k$ that converges geometrically to the solution $u$ in the $L^2$-norm and the broken $H^1$-norm, for any $\theta\in (0,1)$. Moreover, the convergence factor is given by $\vert1-2\theta\vert$, which also proves that this method becomes a direct solver for the specific choice $\theta=\frac{1}{2}$.
    \label{cor:GeoCvgDirDN}
\end{cor}

\begin{remark}
	At each iteration $k$, due to the symmetry properties of $u_e^k$ and $u_o^k$, it is enough to solve the subproblems in $\Omega_1$ and $\Omega_2$ only. This implies that our method does not even require additional computational work compared to the original Dirichlet-Neumann method. Indeed, solving seperately for the even and odd symmetric parts in two subdomains costs the same as solving for the whole solution in four subdomains.
\end{remark}

% -------------------------------------
\section{Extension to three dimensions}
\label{sec:3D}
% -------------------------------------

In this section, we show that the interesting properties of our new DN method remain valid in three dimensions.
Indeed, in the proofs of the previous results, the key ingredients are the symmetry properties of the functions with respect to the cartesian coordinate system. Therefore, the extension of these results to three dimensions seems quite natural. Let us consider the cube $\Omega = (-1,1)^3\subset \mathbb{R}^3$ divided into four subdomains as shown in Figure \ref{fig:Sch3D4subd}.

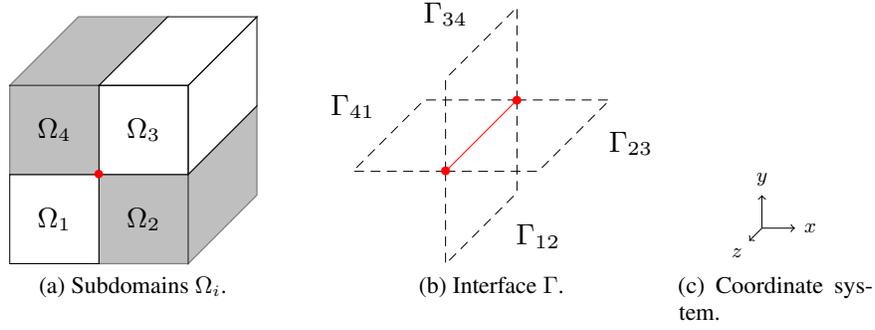
\begin{figure}
\begin{center}
\subfloat[Subdomains $\Omega_i$.]{   
\resizebox{!}{0.2\textwidth}{
\begin{tikzpicture}

	\pgfmathsetmacro{\cubex}{1}
	\pgfmathsetmacro{\cubey}{1}
	\pgfmathsetmacro{\cubez}{1}

	% Left bottom subdomain (Omega 1)
	\draw[very thin] (0,0,\cubez) -- ++(-\cubex,0,0) -- ++(0,-\cubey,0) -- ++(\cubex,0,0) -- cycle;
	% Right bottom subdomain (Omega 2)
	\draw[fill=gray,opacity=0.5] (\cubex,0,\cubez) -- ++(-\cubex,0,0) -- ++(0,-\cubey,0) -- ++(\cubex,0,0) -- cycle;
	\draw[fill=gray,opacity=0.5] (\cubex,0,\cubez) -- ++(0,0,-2*\cubez) -- ++(0,-\cubey,0) -- ++(0,0,2*\cubez) -- cycle;
	% Top right subdomain (Omega 3)
	\draw[very thin] (\cubex,\cubey,\cubez) -- ++(-\cubex,0,0) -- ++(0,-\cubey,0) -- ++(\cubex,0,0) -- cycle;
	\draw[very thin] (\cubex,\cubey,\cubez) -- ++(0,0,-2*\cubez) -- ++(0,-\cubey,0) -- ++(0,0,2*\cubez) -- cycle;
	\draw[very thin] (\cubex,\cubey,\cubez) -- ++(-\cubex,0,0) -- ++(0,0,-2*\cubez) -- ++(\cubex,0,0) -- cycle;
	% Top left subdomain (Omega 4)
	\draw[fill=gray,opacity=0.5] (0,\cubey,\cubez) -- ++(-\cubex,0,0) -- ++(0,-\cubey,0) -- ++(\cubex,0,0) -- cycle;
	\draw[fill=gray,opacity=0.5] (0,\cubey,\cubez) -- ++(-\cubex,0,0) -- ++(0,0,-2*\cubez) -- ++(\cubex,0,0) -- cycle;

    \node[] at (-0.5*\cubex,-0.5*\cubey,\cubez) {\small{$\Omega_1$}};
    \node[] at (0.5*\cubex,-0.5*\cubey,\cubez) {\small{$\Omega_2$}};
    \node[] at (0.5*\cubex,0.5*\cubey,\cubez) {\small{$\Omega_3$}};
    \node[] at (-0.5*\cubex,0.5*\cubey,\cubez) {\small{$\Omega_4$}};
     
    \node[red] at (0,0,\cubez) {\tiny{$\bullet$}};
    
\end{tikzpicture}
}
}
\hspace{1em}
\subfloat[Interface $\Gamma$.]{   
\resizebox{!}{0.22\textwidth}{
\begin{tikzpicture}

	\pgfmathsetmacro{\cubex}{1}
	\pgfmathsetmacro{\cubey}{1}
	\pgfmathsetmacro{\cubez}{1}

	% Left interface (Gamma 41)
	\draw[very thin,densely dashed] (0,0,\cubez) -- ++(-\cubex,0,0) -- ++(0,0,-2*\cubez) -- ++(\cubex,0,0) ;
	% Right interface (Gamma 23)
	\draw[very thin,densely dashed] (0,0,\cubez) -- ++(\cubex,0,0) -- ++(0,0,-2*\cubez) -- ++(-\cubex,0,0) ;
	% Bottom interface (Gamma 12)
	\draw[very thin,densely dashed] (0,0,\cubez) -- ++(0,-\cubey,0) -- ++(0,0,-2*\cubez) -- ++(0,\cubey,0) ;
	% Top interface (Gamma 34)
	\draw[very thin,densely dashed] (0,0,\cubez) -- ++(0,\cubey,0) -- ++(0,0,-2*\cubez) -- ++(0,-\cubey,0) ;	
	% Cross edge
	\draw[very thin,red] (0,0,\cubez) -- (0,0,-\cubez) ;

    \node[] at (-0*\cubex,1.7*\cubey,\cubez) {\small{$\Gamma_{34}$}};
    \node[] at (1*\cubex,-0.7*\cubey,\cubez) {\small{$\Gamma_{12}$}};
    \node[] at (2*\cubex,0.3*\cubey,\cubez) {\small{$\Gamma_{23}$}};
    \node[] at (-1*\cubex,0.7*\cubey,\cubez) {\small{$\Gamma_{41}$}};

    \node[red] at (0,0,\cubez) {\tiny{$\bullet$}};
    \node[red] at (0,0,-\cubez) {\tiny{$\bullet$}};  
      
\end{tikzpicture}
}
}
\subfloat[Coordinate system.]{ 
\hspace{1em}
\resizebox{!}{0.08\textwidth}{   
\begin{tikzpicture}

	\pgfmathsetmacro{\cubex}{1}
	\pgfmathsetmacro{\cubey}{1}
	\pgfmathsetmacro{\cubez}{1}
	
    % Coordinate system
    \coordinate (o) at (0,-0.5\cubey,0);
    \draw[->] (o) --++ (0.5\cubex,0,0);
    \draw[->] (o) --++ (0,0.5\cubey,0);
    \draw[->] (o) --++ (0,0,0.5\cubez);
    \node[] at (0.75*\cubex,-0.5*\cubey,0) {\small{$x$}};
    \node[] at (0,0.25\cubey,0) {\small{$y$}};
    \node[] at (0,-0.5*\cubey,\cubez) {\small{$z$}}; 
   
\end{tikzpicture}
}
\hspace{1em}
}
\end{center}
\caption{Domain $\Omega$ divided into four subdomains.}
\label{fig:Sch3D4subd}
\end{figure}

In this specific configuration, we keep the same geometry as in the bidimensional case, and stretch it along the $z$ direction. This leads to a configuration with one interior "cross-edge" consisting of the points $(x,y,z)\in\{0\}\times\{0\}\times(-1,1)$, drawn in red in Figure \ref{fig:Sch3D4subd}, and four boundary cross-edges. Since there is no symmetry in $z$ in this geometric configuration, we consider even/odd symmetry with respect to $(x,y)$ instead of real three-dimensional even/odd symmetry. In order to avoid introducing new terminology, we will talk about the "even" and "odd" symmetric parts of the problems when referring to even/odd symmetry with respect to $(x,y)$, even though this is technically not precise. Hence, the even/odd decomposition considered in the three-dimensional case with four subdomains is the following: for any function $h$ in $L^p(\Omega$), we introduce for a.a. $(x,y,z)\in\Omega$
\begin{equation*}
	\begin{aligned}
		h_{\tilde{e}}(x,y,z) &:= \frac{1}{2}\left( h(x,y,z)+h(-x,-y,z) \right)\:, \\
		h_{\tilde{o}}(x,y,z) &:= \frac{1}{2}\left( h(x,y,z)-h(-x,-y,z) \right)\:,
	\end{aligned}
\end{equation*}
and use the splitting $h=h_{\tilde{e}}+h_{\tilde{o}}$ in order to define our new Dirichlet-Neumann method.

The definition of a compatible initial guess remains valid in this context. Only, here, the pointwise compatibility relations at boundary cross-points become equalities on the boundary cross-edges. Besides, if we introduce $\omega:=(-1,1)^2\subset \mathbb{R}^2$ such that $\Omega=\omega\times (-1,1)$, the tangential vector $\tau$ from Definition \ref{def:CompatibleIG} must be understood as the tangential vector to $\partial\omega$.
Moreover, in order for the proof of the 2D case to be applicable here, we need the traces of $u$ and $u^0$ on the plane $\{z=c\}$ to belong to $W^{2,p}(\omega\times \{c\})$, for any $c\in(-1,1)$. The trace theorem provides us with a necessary condition for this property to hold: $u$ and $u^0$ must belong to $W^{2+\frac1p,p}(\Omega)$.

\vspace{0.5em}
\textbf{Assumption 2.} The solution $u$ to \eqref{eqn:ModelPb} belongs to $ W^{2+\frac1p,p}(\Omega)$.

\begin{theorem}
	Under Assumption 2, for any $u^0\in W^{2+\frac1p,p}(\Omega)$ compatible with the boundary condition, our new Dirichlet-Neumann method applied to \eqref{eqn:ModelPb} in three-dimensions with four subdomains produces a sequence $\{u^k\}_k$ that converges geometrically to the solution $u$ in the $L^2$-norm and the broken $H^1$-norm, for any $\theta\in (0,1)$. Moreover, the convergence factor is given by $\vert 1-2\theta\vert$, which also proves that this method becomes a direct solver for the specific choice $\theta=\frac{1}{2}$.
    \label{thm:GeoCvgDirDN3D}
\end{theorem}

\begin{proof}
	 This can be proved using separation of variables in the $z$ direction in order to recover the two-dimensional case and apply Corollary \ref{cor:GeoCvgDirDN}. As in the proofs of Theorem \ref{thm:GeoCvgEvenDirDN} and Theorem \ref{thm:GeoCvgOddDirDN}, we write the problem in terms of the local errors $e_i^k$ for $i\in\mathcal{I}$ and $k\geq 1$. Due to the geometry of the domain decomposition, each local error is harmonic in $\Omega_i$ and it satisfies on the $z$ faces
\begin{equation*}
	\mathcal{B}e_i^k(x,y,-1) = \mathcal{B}e_i^k(x,y,1) = 0\:, \quad \forall (x,y) \in \omega\:.
\end{equation*}
Using the separation of variables approach, we search for solutions under the form $E(x,y)F(z)$. This leads to solving a Sturm-Liouville problem for $F$. The associated pairs of eigenvalues/eigenfunctions $(\lambda_m,v_m)$ for $m\geq 1$ are such that the $v_m$ form a orthogonal basis of $L^2(-1,1)$, which enables us to write $e_i^k$ using a series expansion
\begin{equation*}
	e_i^k(x,y,z) = \sum_{m\geq 1} \widehat{e}_i^k[m](x,y) v_m (z) \quad \mbox{ for a.e. } (x,y,z)\in\Omega \:.
\end{equation*}
Therefore, rewriting the equation $-\Delta e_i^k = 0$ in $\Omega_i$ using that $-v_m''=\lambda_m v_m$ in $(-1,1)$, we end up with
\begin{equation*}
	-\Delta \widehat{e}_i^k + \lambda_m\widehat{e}_i^k = 0 \: \mbox{ in } \: \omega_i := \Omega_i\cap\{z=0\}\:, \quad \forall m\geq 1\:,
\end{equation*}
where we dropped the dependance on $m$ in $\widehat{e}_i^k $ for simplicity.
On the other hand, given that we consider even/odd symmetry with respect to $(x,y)$ only, we obtain for $e_i^k$ the decomposition
\begin{equation*}
		e_i^k = e_{\tilde{e},i}^k + e_{\tilde{o},i}^k  = \sum_{m\geq 1} \widehat{e}_{e,i}^k v_m + \sum_{m\geq 1} \widehat{e}_{o,i}^k v_m \:.
\end{equation*}
It follows that studying the error $e^k$ for the 3D modified Dirichlet-Neumann method amounts to studying the error $\widehat{e}^k$ for the 2D modified Dirichlet-Neumann method for all $m\geq 1$. For each $m\geq 1$, the problems solved by the local errors $\widehat{e}_{e,i}^k$ and $\widehat{e}_{o,i}^k$ are the same as the ones solved by the $e_{e,i}^k$ and $e_{o,i}^k$ in the 2D case, except from the following modifications:
\begin{itemize}
	\item[(i)] the operator $-\Delta$ is replaced by $-\Delta+\lambda_m$;
	\item[(ii)] the local subdomains $\Omega_i$ are replaced by $\omega_i$;
	\item[(iii)] the local interfaces $\Gamma_{ij}$ are replaced by $\gamma_{ij}:=\Gamma_{ij}\cap\{z=0\}$.
\end{itemize}
First, let us recall that the $\lambda_m$ are non-negative for each $m\geq 1$, therefore for each $i\in\mathcal{I}$ the local boundary value problem associated to the operator $-\Delta+\lambda_m$ is well-posed in $\omega_i$. In addition, the following properties hold: the operator $-\Delta+\lambda_m$ preserves symmetry in the sense of \cite[Definition 3]{chaudet2022cross1}, $u\vert_\omega\in W^{2,p}(\omega)$, and $u^0\vert_\omega$ is compatible with the boundary conditions on $\partial\omega$. Thus the arguments in the proofs of Theorem \ref{thm:GeoCvgEvenDirDN} and Theorem \ref{thm:GeoCvgOddDirDN} apply. And we obtain the expected behaviours for the even/odd recombined errors in $\omega$ for each $m\geq 1$, namely
\begin{equation*}
	\widehat{e}_{e}^k = (1-2\theta)^{k-1}\widehat{e}_{e}^1 \quad \mbox{ and } \quad \widehat{e}_{o}^k = (1-2\theta)^{k-1}\widehat{e}_{o}^1 \:.
\end{equation*}
Therefore, by linearity, we finally get
\begin{equation*}
	e^k = (1-2\theta)^{k-1} \sum_{m\geq 1} \widehat{e}^1 v_m = (1-2\theta)^{k-1} e^1\:,
\end{equation*}
which yields the desired result.
\end{proof}

\begin{remark}
	Note that, as in the two-dimensional case, the standard Dirichlet-Neumann method applied to the odd-symmetric part of \eqref{eqn:ModelPb} is not well-posed and it generates iterates that are singular in the neighbourhood of the cross-edge. This can be proved using the separation of variables introduced in the previous proof. As mentionned above, studying the odd symmetric 3D case amounts to studying the 2D case for the local iterates $\widehat{e}^k_{o,i}$, for all $m\geq 1$. We can thus follow the same steps as in the proof of \cite[Theorem 7]{chaudet2022cross1}, and conclude that for each $m$, there exists some iteration $k_0$ such that the iterates $\widehat{e}^{k_0}_{o}$ are not unique, and that all possible iterates are singular near the cross-point, with a leading singularity of type $(\ln r)^2$. 
%	More specifically, if $\mathcal{V}_0$ denotes a neighbourhood of the cross-point, we get in $\omega_2\cap\mathcal{V}_0$ for the local error expressed in polar coordinates $(r,\phi)$:
%\begin{equation*}
%	\widehat{e}^{k_0}_{o,2} \simeq -\theta \widehat{\delta}^{k_0-1}[m]\frac{8}{\pi^2}(\ln r)^2 \:,
%\end{equation*}
%where $\widehat{\delta}^{k_0-1}[m]:=\widehat{e}^{k_0-1}_{o,2}(0,0)[m]$. Combining these expressions for all frequencies $m$, we obtain in $\Omega_2\cap (\mathcal{V}_0\times (-1,1))$ for the local error expressed in cylindrical coordinates $(r,\phi,z)$:
%\begin{equation*}
%	e^{k_0}_{o,2} \simeq -\theta\frac{8}{\pi^2}(\ln r)^2 \sum_{m\geq 1} \widehat{\delta}^{k_0-1}[m]v_m(z) =: -\theta\frac{8}{\pi^2}(\ln r)^2 \delta^{k_0-1}(z) \:,
%\end{equation*}
%where $\delta^{k_0-1}$ is function of $z$ that belongs to $L^2(-1,1)$.
\end{remark}

\section{Numerical experiments}
\label{sec:Numerics}
% -----------------------------

We now illustrate our theoretical results with numerical experiments in two and three dimensions. The domain $\Omega$ is discretized using a regular grid of size $h$, and our numerical method is based on a standard five-point finite difference scheme in the 2D case, and on its nine-point version in the 3D case. In problems with mixed boundary conditions (Dirichlet and Robin), the Dirichlet boundary condition is enforced weakly using a penalty parameter of order $10^{-12}$. Unless otherwise stated, the mesh size will be set to $h=2\cdot10^{-2}$ in all 2D experiments, and to $h=6\cdot10^{-2}$ in all 3D experiments. In addition, for a given mesh size $h$, the relative numerical error is computed taking $u_{ex}$ as the discrete solution obtained when solving on the whole domain $\Omega$ with a direct solver.

\subsection{Example 1}

This first example is to illustrate the result of
Theorem \ref{thm:GeoCvgEvenDirDN}, i.e. convergence for the even symmetric part in the 2D case. We consider the square domain $\Omega=(-1,1)^2$, and we take the source term $f=1$ in $\Omega$, with symmetric mixed Dirichlet-Robin boundary conditions: $u=0$ on $\partial\Omega_l\cup\partial\Omega_r$ and $\partial_n u=0$ on $\partial\Omega_b\cup\partial\Omega_y$. A simple compatible initial guess (in the sense of Definition \ref{def:CompatibleIG}) in this case is $u^0=0$ in $\Omega$. 
\begin{figure}
  \centering
    \subfloat[Error at iteration 2, $\theta=0.5$.]{
    \includegraphics[width=0.47\textwidth]{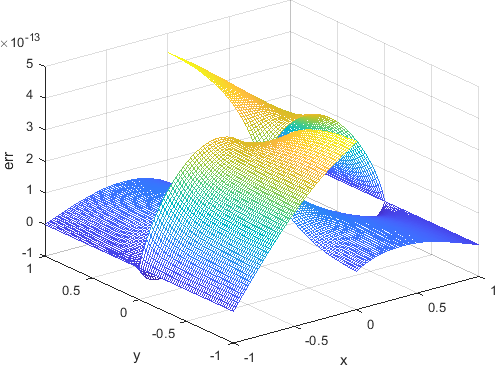}
    \label{sub:ErrTh05Ex1}
    }
    %\hspace{.1em}
    \subfloat[$L^2$-norm of the error.]{
    \includegraphics[width=0.47\textwidth]{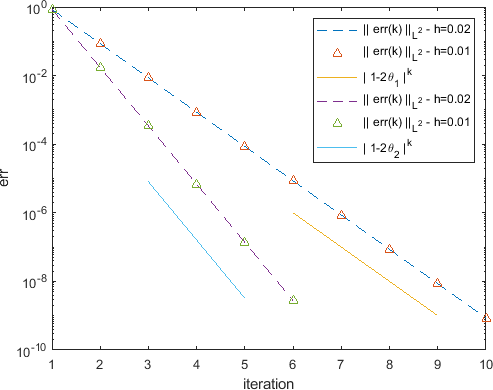}
    \label{sub:CvgTh045Ex1}
    }
    \hspace{.1em}
    \subfloat[Error at iteration 1, $\theta=0.45$.]{
    \includegraphics[width=0.47\textwidth]{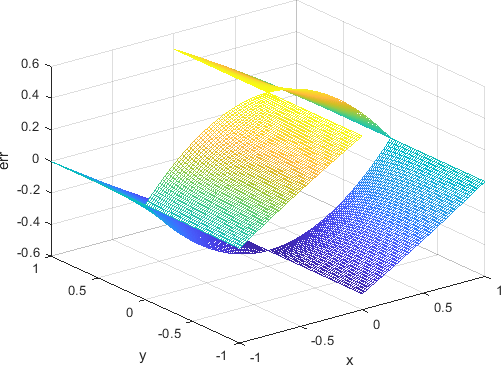}
    \label{sub:Err1Th045Ex1}
    }
    %\hspace{.1em}
    \subfloat[Error at iteration 7, $\theta=0.45$.]{
    \includegraphics[width=0.47\textwidth]{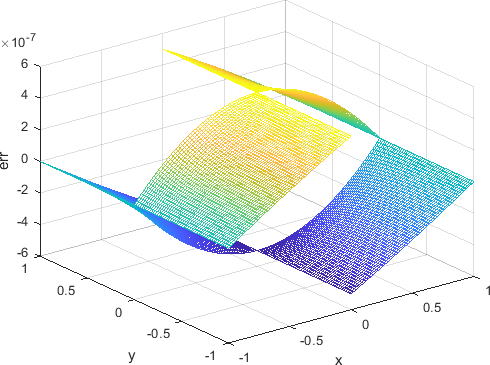}
    \label{sub:Err7Th045Ex1}
    }
  \caption{Results for the new DN method applied to \eqref{subeqn:EvenModelPb} in 2D (Example 1).}
  \label{fig:Example1}
\end{figure}
The results confirm that, for $\theta=\frac{1}{2}$, the new DN method is a direct
solver. Indeed, Figure \ref{sub:ErrTh05Ex1} shows that the error at iteration 2 is "zero" (here it cannot be
much smaller than $10^{-12}$ since we penalize Dirichlet boundary conditions). In the case $\theta\neq\frac{1}{2}$, we see from Figure \ref{sub:Err1Th045Ex1} and Figure \ref{sub:Err7Th045Ex1} that the error in $\Omega$ is multiplied by a constant from one iteration to the next, which is consistent with the recurrence relation in the proof of \cite[Theorem 6]{chaudet2022cross1}. We also see that, as predicted by the proof of Theorem \ref{thm:GeoCvgEvenDirDN}, at each iteration, the error remains continuous at the cross-point from $\Omega_1$ to $\Omega_3$, and from $\Omega_2$ to $\Omega_4$, and it has the expected symmetry property. Finally, we also observe that the new DN method converges geometrically with the expected convergence factor, see Figure \ref{sub:CvgTh045Ex1} where we have taken different values for $\theta$: $\theta_1=0.45$ and $\theta_2=0.49$. Also note that, as predicted by the theory, the convergence behaviour of the method does not depend on the meshsize $h$. Indeed, we can see on the graph that the error curves for $h=0.02$ and $0.01$ are almost overlaid on each other.

\subsection{Example 2}

We illustrate now the result of Theorem \ref{thm:GeoCvgOddDirDN} for the odd symmetric part in the 2D case. We choose the same domain  and boundary conditions as in Example 1, but here we set the source term $f=\sin(\pi x)\cos(\frac{\pi}{2}y)$ in $\Omega$. The same compatible initial guess $u^0=0$ in $\Omega$ is considered. As shown in Figure \ref{fig:Example2},
\begin{figure}
  \centering
    \subfloat[Error at iteration 2, $\theta=0.5$.]{
    \includegraphics[width=0.47\textwidth]{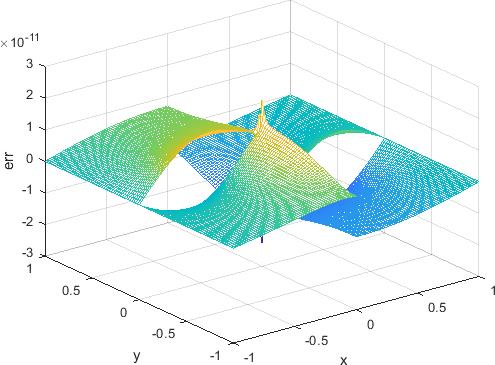}
    \label{sub:ErrTh05Ex2}
    }
    %\hspace{.1em}
    \subfloat[$L^2$-norm of the error.]{
    \includegraphics[width=0.47\textwidth]{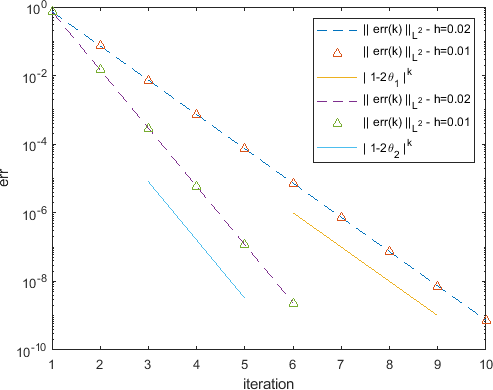}
    \label{sub:CvgTh045Ex2}
    }
    \hspace{.1em}
    \subfloat[Error at iteration 1, $\theta=0.45$.]{
    \includegraphics[width=0.47\textwidth]{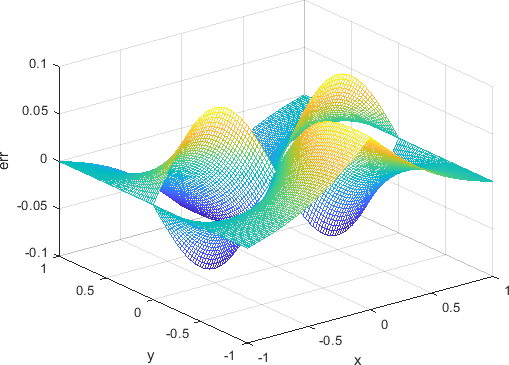}
    \label{sub:Err1Th045Ex2}
    }
    %\hspace{.1em}
    \subfloat[Error at iteration 7, $\theta=0.45$.]{
    \includegraphics[width=0.47\textwidth]{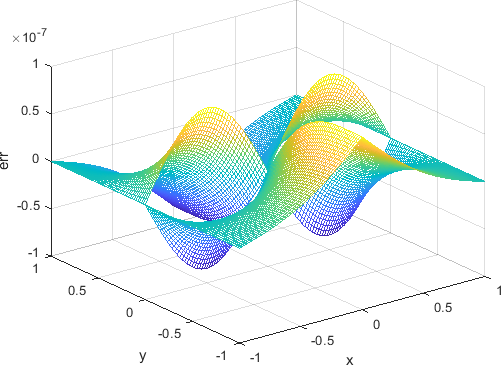}
    \label{sub:Err7Th045Ex2}
    }
  \caption{Results for the new DN method applied to \eqref{subeqn:OddModelPb} in 2D (Example 2).}
  \label{fig:Example2}
\end{figure}
the new transmission conditions proposed in Section \ref{sec:NewDN} enable us to recover the same convergence behaviour as for the even symmetric case. In particular, the new DN method becomes a direct solver when $\theta$ is set to $\frac{1}{2}$ (see Figure \ref{sub:ErrTh05Ex2}), and for other choices of $\theta$, it converges geometrically with the expected convergence factor $(1-2\theta)$, as shown in Figure \ref{sub:CvgTh045Ex2} for $\theta_1=0.45$ and $\theta_2=0.49$. Moreover, the convergence behaviour does not depend on $h$. Figures \ref{sub:Err1Th045Ex2} and \ref{sub:Err7Th045Ex2} also illustrate the recurrence relation and the symmetry properties of the error obtained in the proof of Theorem \ref{thm:GeoCvgOddDirDN}.

\subsection{Example 3}

In order to validate the results stated in Theorem \ref{thm:GeoCvgDirDN3D} for the 3D case, we consider the cube $\Omega=(-1,1)^3$ divided into four subdomains as in Figure \ref{fig:Sch3D4subd}. In this first 3D-example, we study the even symmetric case and choose $f=1$ in $\Omega$ with a homogeneous Dirichlet boundary condition $u=0$ everywhere on $\partial\Omega$. We start the iterative process with the compatible intial guess $u^0=0$ in $\Omega$.
\begin{figure}
  \centering
    \subfloat[Error at iteration 2, $\theta=0.5$.]{
    \includegraphics[width=0.47\textwidth]{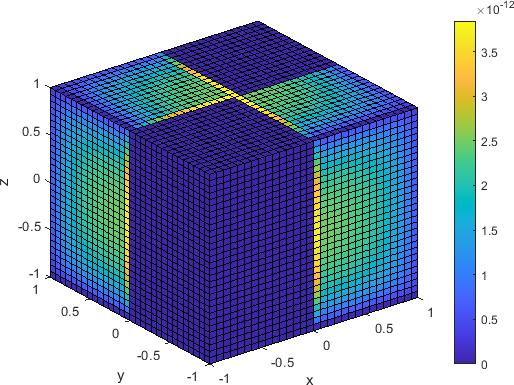}
    \label{sub:ErrTh05Ex3}
    }
    %\hspace{.1em}
    \subfloat[$L^2$-norm of the error.]{
    \includegraphics[width=0.47\textwidth]{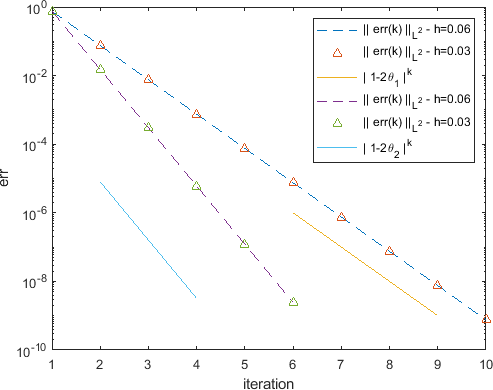}
    \label{sub:CvgTh045Ex3}
    }
  \caption{Results for the new DN method applied to \eqref{subeqn:EvenModelPb} in 3D (Example 3).}
  \label{fig:Example3}
\end{figure}
Similarly to the 2D case, we get that for $\theta=\frac12$, the error reduces to zero after two iterations, see Figure \ref{sub:ErrTh05Ex3}. Here, the error is actually exactly zero in $\Omega_1$ and $\Omega_3$ since the subproblems associated to these subdomains only have Dirichlet boundary conditions, which are thus enforced exactly and not using a penalty technique. Furthermore, for $\theta\neq\frac12$, the new DN method converges geometrically with the expected convergence factor, and the convergence is independent of the meshsize, as shown in Figure \ref{sub:CvgTh045Ex3} for $\theta_1=0.45$, $\theta_2=0.49$ and $h\in\{ 0.03, 0.06\}$.

\subsection{Example 4}

This last example is dedicated to the odd symmetric part in the 3D case. Therefore, we set $f=\sin(\pi x)y^2z$ in $\Omega$, and the same homogeneous Dirichlet boundary condition as in Example 3, which allows us to keep $u^0=0$ as compatible initial guess. 
\begin{figure}
  \centering
    \subfloat[Error at iteration 2, $\theta=0.5$.]{
    \includegraphics[width=0.47\textwidth]{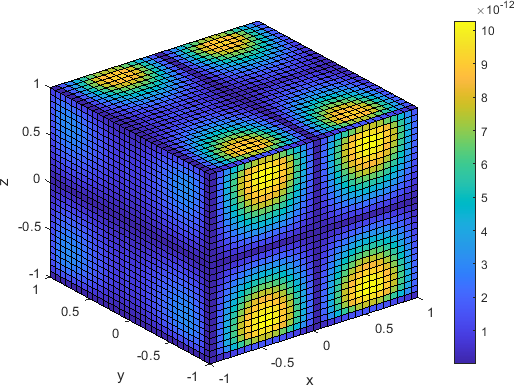}
    \label{sub:ErrTh05Ex4}
    }
    %\hspace{.1em}
    \subfloat[$L^2$-norm of the error.]{
    \includegraphics[width=0.47\textwidth]{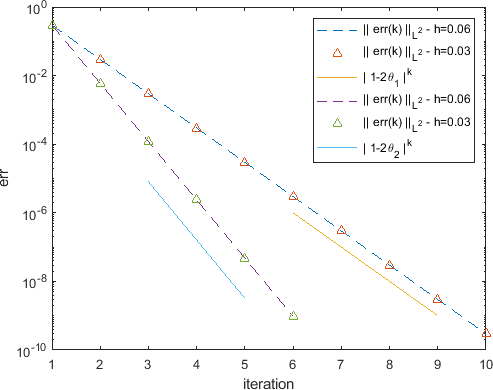}
    \label{sub:CvgTh045Ex4}
    }
  \caption{Results for the new DN method applied to \eqref{subeqn:OddModelPb} in 3D (Example 4).}
  \label{fig:Example4}
\end{figure}
In this case as well, the convergence properties predicted by the theory are observed numerically, as shown in Figure \ref{sub:ErrTh05Ex4} for $\theta=\frac12$ and Figure \ref{sub:CvgTh045Ex4} for $\theta\neq\frac12$, where again $\theta_1=0.45$ and $\theta_2=0.49$.

\section{Conclusion}

In this paper, we completed the analysis of the DN method started in \cite{chaudet2022cross1}. First, we showed that the idea of an even/odd symmetric decomposition could be extended to the case of mixed Dirichlet-Robin boundary conditions, and to the three-dimensional case. Based on this decomposition, we proved that, for the even symmetric part of the problem, the original DN method was geometrically convergent with a convergence factor independent of $h$ in all aforementioned cases, which generalizes the results in \cite{chaudet2022cross1}. Then, we introduced a new variant of the DN method based on a different distribution of the Dirichlet/Neumann transmission conditions and specifically tuned for dealing with odd symmetric functions. For the odd symmetric part of the problem, we proved that this new variant has the same convergence properties as the original DN method applied to the even symmetric part. Finally, we illustrated our theoretical results with numerical experiments in two and three dimensions.

A first direction of future work is to study how the new DN method extends to partitions with more than four square subdomains. Also, since the analysis presented here is limited to the case of rectilinear cross-points, it would be interesting to generalize it to more complicated cross-points, with non right angles and possibly involving a number of subdomains $N\neq 4$.

\bibliographystyle{acm}
\bibliography{references}

\end{document}